%% file: article_ms_scattering.tex
\newif\ifPREPRINT

\PREPRINTtrue

\ifPREPRINT
\documentclass[a4paper,11pt]{amsart}

\usepackage{a4wide}

\usepackage[foot]{amsaddr}

\else
\documentclass{m2an}
\fi

\usepackage[breaklinks,bookmarks=false]{hyperref}
\hypersetup{%
colorlinks,%
linkcolor=blue,%
citecolor=blue,%
urlcolor=blue,%
plainpages=false,%
pdfwindowui=false,%
pdfstartview={FitH}%
}

\usepackage{pgfplots}
\usepackage{subcaption}
\usepackage{mathrsfs}
\usepackage{amssymb}

\input{general_commands}

\input{specific_commands}
\input{slope_triangle}

\begin{document}

\ifPREPRINT
\title[A multiscale finite element method for scattering problems]%
{Frequency-explicit convergence analysis of a multiscale
finite element method for highly heterogeneous scattering problems}

\author{T. Chaumont-Frelet$^\dagger$ and Z. Kassali$^\star$}

\address{\vspace{-.5cm}}
\address{\noindent \tiny \textup{$^\dagger$Inria, Univ. Lille, CNRS, UMR 8524 -- Laboratoire Paul Painlev\'e}}
\address{\noindent \tiny \textup{$^\star$Chemnitz University of Technology}}
\fi

\begin{abstract}
We analyze the numerical approximation of time-harmonic scattering
by highly heterogeneous penetrable obstacles. These problems are especially
challenging in the high-frequency regime, where the size of the scatterer $\hD$
is much larger than the wavelength, i.e., the wavenumber $k$ is such that $k\hD \gg 1$.
Here, we further consider the situation where the scatterer contains different
materials, with a characteristic size $\eps$ such that $k\eps \ll 1$. We propose
a high-order multiscale finite element method, and provide an error analysis that
is explicit in both $k$ and $\eps$. Crucially, our error estimates suggest that using a
high-order method should reduce the computational cost for large frequencies, which is
corroborated by numerical examples.
\ifPREPRINT

\vspace{.5cm}
\fi
\end{abstract}

\maketitle

\section{Introduction}

We consider high-frequency wave propagation problem in highly heterogeneous media.
Specifically, we consider situation where the typical wavelength of the solution
is much smaller than the size of the computational domain, but still much larger
that the characteristic length of the heterogeneities in the propagation medium.
We focus in particular on scattering problems model by the Helmholtz equation
with wavenumber $k$ where the coefficients inside the scatterer are periodic
with a small period $\varepsilon$. Denoting by $\hD$ the diameter of the scatterer,
we therefore focus on the regime where $k\eps \ll 1$ and $k\hD \gg 1$.

The first contribution of this work is to derive homogenization error estimates
which can be made fully-explicit in the wavenumber in some situations. These
results generalises existing ones in~\cite{bouchitte_felbacq_2004a}
and~\cite{chaumontfrelet_spence_2025a}. In particular, a key novelty of the present work
is to show that there exists a constant $\Cst$ depending on $k$ but independent of $\eps$
such that
\begin{equation}
\label{eq_stability_estimate_intro}
k^2 \|u\|_{L^2(\Omega)} \leq \Cst \|f\|_{L^2(\Omega)}
\end{equation}
for all $f \in L^2(\Omega)$ and all $\eps > 0$. Here, $u$ is the solution to
the scattering problem with right-hand side $f$. The proof
of~\eqref{eq_stability_estimate_intro} based on~\cite{bouchitte_felbacq_2004a}
where homogenization results for a single right-hand side $f$ are established,
and we hereby make such results uniform for all $L^2(\Omega)$ right-hand sides.
Under appropriate assumption on the coefficients inside the scatterer, we are
further able to derive full-explicit bounds on $\Cst$
following~\cite{chaumontfrelet_spence_2025a}.

The second theoretical contribution of the present work is to analyze
a multiscale finite element method to discretize the highly heterogeneous
scattering problem. We specifically, focus on the multiscale hybrid-mixed
(MHM) method, where, given a computational mesh with element size $H$,
multiscale basis functions are computed solving Neumann problems on each
mesh element with prescribed polynomial boundary data of degree $\ell \geq 0$.
Although we focus the MHM, the methodology is generic and can, at least in principle,
be applied to other multiscale methods such as
MsFEM~\cite{allaire2005multiscale,hou_convergence_1999,sangalli2003capturing}
where Dirichlet data are prescribed to compute the basis functions. Assuming that the
basis functions are exactly computed, we show that if $\eps \ll H$,
$\Cst^2 (k\eps)^{1/2} \ll 1$ and $\Cst (kH)^{\ell+1} \ll 1$,
then the finite element solution is quasi-optimal, and we have
\begin{equation}
\label{eq_error_estimate_intro}
k\|u-u_H\|_{L^2(\Omega)} + \|\grad(u-u_h)\|_{\BL^2(\Omega)}
\lesssim
kH + \Cst (kH)^{\ell+1} + \Cst^2 (k\eps)^{1/2},
\end{equation}
where $u_H$ is the discrete solution and $\Cst$ is the stability constant
from~\eqref{eq_stability_estimate_intro}. Classically~\cite{hou_convergence_1999},
when $\eps \sim H$, the method suffers from the so-called ``resonance'' effect,
and extra terms involving the ratio $\eps/H$ need to be added
into~\eqref{eq_error_estimate_intro}. The estimate
in~\eqref{eq_error_estimate_intro} generalises standard error estimates
from~\cite{bernkopf2025wavenumber,chaumont-frelet_wavenumber_2020,lafontaine_wavenumber-explicit_2022}
for standard finite element methods in the absence of highly
oscillating coefficient. Crucially, we see that although the stability
constant $\Cst$ grows with the frequency, such growth can be effectively
mitigated by employing high-order elements with a ``large'' value of $\ell$.
The fact that high-order elements are better suited to solve high-frequency
wave propagation problems is now well-established in the literature, but this
contributions is, to the best of our knowledge, the first to establish it for
multiscale problems.

The estimates derived in this work assume exact element-wise PDE solves 
to compute the basis functions. In practice, this is unfeasible, and a
second-level discretization with mesh size $h \sim \eps$ is employed within
in each element. We do employ such a methodology numerically, but we do
not analyse the effect of the second-level discretization in our theoretical
results.

We note that more modern multiscale numerical methods based on localized
orthogonal decomposition (LOD) are able to be robust for large frequencies under
the sole assumption that $kH \ll 1$,
see~\cite{brown_multiscale_2015,gallistl_stable_2015,peterseim_eliminating_2015,peterseim_computational_2020}.
In particular, such methods
do not suffer from the resonant effect when $\eps \sim H$. However, LOD-based methods
require to solve local PDE problems on patches of elements, rather than on single elements,
which can make them much more expensive in practice. In contrast, in the presence
of scale separation where $\eps \ll H$, the methodology proposed here is likely to
be more efficient since it only relies on element-wise PDE solves with a moderate
polynomial degree $\ell$. Similar comments also seem to apply to MsGFEM
from~\cite{ma2023wavenumber}.

Our theoretical findings are complemented with numerical examples.
These examples highlight the key features of the theory, and in particular,
they show that there is indeed an interest in using a high-order multiscale
method for high-frequency problems.

The remainder of this work is organized as follows.
Section~\ref{section_setting} makes the setting precise, recall some basic results
from the literature, and rigorously states all the assumptions we make
on the scattering problem, including the multiscale coefficients.
Section~\ref{section_uniform_stability} is dedicated to the proof
of the uniform stability estimate in~\eqref{eq_stability_estimate_intro}.
In Section~\ref{section_homogenization} we couple this stability result
with standard tools from the literature to obtain frequency-explicit
homogenization error estimates for the scattering problem. Section~\ref{section_mhm}
introduce the MHM method, which is subsequently analyzed in
Section~\ref{section_mhm_discretization} where we established
in~\eqref{eq_error_estimate_intro}. Numerical examples are
finally presented in Section~\ref{section_numerical_examples}.
Finally, Appendix~\ref{appendix_homogenization} contains auxiliary
homogenization results for the element wise PDE problems used to compute
the basis functions in the MHM method.

\section{Setting}
\label{section_setting}

\subsection{Geometrical configuration}

Let $D \subset \R^d$, $d=2,3$, be a smooth bounded domain representing
the obstacle. We assume for simplicity that $D$ contains the origin, and
denote by $\hD$ its diameter.
We also consider another smooth domain $\Om \subset \R^d$ such that
$D \Subset \Om$, and impose (approximate) transparent conditions
on $\partial \Om$. Throughout we denote by $\bn$ the
normal unit vector pointing outward $\Om$. We also note
$\Gamma \eq \partial D$. For simplicity, we also assume
that $d(\partial \Om,D) \geq \hD$.

Inside the obstacle $D$, we will consider highly heterogeneous coefficients,
and $0 < \eps \leq \hD$ will denote their characteristic length scale. Besides,
we will denote by $k > 0$ the wavenumber. We are especially interested in the
regime where $k\eps \ll 1$ and $k\hD \gg 1$.

\begin{remark}
The requirement that $d(\partial \Omega),D) \geq \hD$ is only used for
the lower bounds in Propositions~\ref{proposition_bound_homo}
and~\ref{proposition_bound_eps}, which allows to simplify other estimates.
However, the analysis can be performed, albeit with heavier notation, without
this assumption.
\end{remark}

\subsection{Function spaces}

If $U \subset \mathbb R^d$ is a domain with a Lipschitz boundary,
$L^2(U)$ is the Lebesgue space of square-integrable functions,
and we respectively denote by $\|{\cdot}\|_{L^2(U)}$ and $({\cdot},{\cdot})_U$
its usual norm and inner product. For $m \geq 0$, $H^m(U)$ then stands for the standard
Sobolev space of order $m$. Its semi-norm is denoted by $|{\cdot}|_{H^m(U)}$.
In addition, when $m=1$, we will employ the weighted norm
\begin{equation*}
\|{\cdot}\|_{H^1_k(U)}^2 \eq k^2\|{\cdot}\|_{L^2(U)}^2 + |{\cdot}|_{H^1(U)}^2.
\end{equation*}
We employ similar notation, but with boldface, for spaces of vector fields.
$H^{1/2}(\partial U)$ is the space of traces of $H^1(U)$, and we denote by
$H^{-1/2}(\partial U)$ its (anti)-dual. The notation $H^1(\Om \setminus \Gamma)$
and $H^2(\Om \setminus \Gamma)$ are respectively used for functions that whose
restriction to both $D$ and $\Om \setminus \overline{D}$ are $H^1$ and $H^2$
regular. We employ the same notation as above for the corresponding norm
and semi-norm, whereby the differential operators are applied independently
in $D$ and $\Om \setminus \overline{D}$. We refer the reader
to~\cite{adams_fournier_2003a} for an introduction to these spaces.

We also employ the standard notation $\BH_0(\ddiv,U)$ for the Sobolev
functions of vector field $\bq \in \BL^2(U)$ with $\div \bq \in L^2(\Om)$
and vanishing normal trace $\partial U$. This space is described in
depth in~\cite{girault_raviart_1986a}.

We denote by $Y \eq (0,1)^d$ the unit cube, which we will use as a periodic cell
to mathematically describe multiscale coefficients. Then $W^{1,\infty}_\sharp(Y)$
is the set of periodic Lipschitz continuous functions, and we employ the notation
$\underline{\BW}^{1,\infty}_\sharp(Y)$ for symmetric matrix-valued functions which are component-wise
$W^{1,\infty}_\sharp(Y)$. We also employed $\BW^{1,\infty}_\sharp(Y)$ for vector-fields.
The norms and seminorms of this spaces are denoted by $\|{\cdot}\|_{W^{1,\infty}(Y)}$
and $|{\cdot}|_{W^{1,\infty}(Y)}$ with obvious modifications for vectors and matrices.
We also employ the notation $H^1_\sharp(Y)$ for the subspace of $H^1(Y)$ with periodic
boundary conditions. A detailed description of spaces with periodic boundary conditions
on $Y$ can be found in~\cite{cioranescu_donato_1999a}.

\subsection{Periodic patterns}
\label{section_periodic_patterns}

Throughout, we consider fixed smooth periodic patterns
$\widehat{\mu} \in W^{1,\infty}_\sharp(Y)$ and
$\widehat{\MA} \in \MW_\sharp^{1,\infty}(Y)$.
We assume standard ellipticity and boundedness conditions, namely that
\begin{equation*}
0
<
\Amin \leq \widehat{\MA}(\by) \bxi \cdot \bxi \leq \Amax
<
\infty  
\qquad \forall \bxi \in \R^d \text{ with } |\bxi| = 1
\end{equation*}
and
\begin{equation*}
0 < \mu_{\min} \leq \widehat{\mu}(\by) \leq \mu_{\max} < +\infty
\end{equation*}
for a.e. $\by \in Y$. The authors believe that the regularity
requirements on $\widehat{\mu}$ and $\widehat{\MA}$ might be alleviated,
but at the price of extra non-trivial technicalities.

\subsection{Hidden constants}
\label{section_hidden_constants}

For two real numbers $A,B$, we employ the notation $A \lesssim B$
if there exists a constant $C$ independent of $A$ and $B$ such that
$A \leq CB$. We also similarly write $A \gtrsim B$ whenever $B \lesssim A$.
In these notation, the constant $C$ is allowed to depend on
$\widehat{\mu}$ and $\widehat{\MA}$, the shape of the obstacle
$D$ and the shape of $\Om$, as well as on $\max(k\eps,1)$ and
$\min(k\hD,1)$. These last two dependencies are allowed to simplify
expressions using that $k\eps \lesssim (k\eps)^{1/2}$ and $k\hD \gtrsim 1$.
In Section~\ref{section_mhm}, we additionally allow the hidden
constant $C$ to depend on the mesh shape-regularity measure $\CK$
introduced in Section~\ref{section_curved_mesh} and on
the polynomial degree $\ell$.
In all these cases, the hidden constant is always independent
of $\eps$, $k$, and the minimal and maximal mesh sizes $H_{\min}$
and $H_{\max}$.

\subsection{Oscillating functions}

We now introduce some useful notation for highly oscillatory functions.
For a vector space $\mathbb{K}$, given
$\widehat \phi: D \times Y \to \mathbb{K}$, we define the corresponding
oscillating function $\phi^\eps: D \to \mathbb{K}$ by
\begin{equation*}
\phi^\eps (\bx) \eq \widehat \phi\left (\bx,\left \{\frac{\bx}{\eps}\right \}\right )
\end{equation*}
for all $\bx \in D$, where $\{t\} \eq t - \lfloor t \rfloor$ denotes the fractional
part of $t \in \R$, and is applied component-wise to vectors. We also employ the notation
\begin{equation*}
\mean{\widehat{\phi}}(\bx) \eq \int_Y \widehat{\phi}(\by)d\by,
\end{equation*}
for $\bx \in D$, and we note that the following identity holds true
\begin{equation}
\label{eq_delta_derivatives}
\pd{\widehat{\phi}^\eps}{\bx_j}
=
\left (\pd{\widehat{\phi}}{\bx_j}\right )^\eps
+
\frac{1}{\eps}
\left (\pd{\widehat{\phi}}{\by_j}\right )^\eps
\end{equation}
for $1 \leq j \leq d$. We also employ the same notation for periodic patterns
$\widehat \phi: Y \to \mathbb{K}$. In particular, we
will consider the coefficients $\widehat{\mu}^\varepsilon$
and $\widehat{\MA}^\varepsilon$ inside the obstacle $D$.

\subsection{Highly heterogeneous scattering problem}

With this notation, we can now introduce our model scattering problem.
We first set the following shorthand notation for the coefficients
\begin{equation*}
\mu_\eps \eq \left \{
\begin{array}{ll}
\widehat{\mu}^\eps & \text{ in } D,
\\
1 & \text{ in } \Omega \setminus D,
\end{array}
\right .
\qquad
\MA_\eps \eq \left \{
\begin{array}{ll}
\widehat{\MA}^\eps & \text{ in } D,
\\
\tens{I} & \text{ in } \Omega \setminus D.
\end{array}
\right .
\end{equation*}
	
Then, for $k > 0$ and given $f \in L^2(\Om)$,
we then consider the following Helmholtz problem: Find
$u_\eps \in H^1(\Om)$ such that
\begin{equation*}
\left \{
\begin{array}{rcll}
- k^2 \mu_\eps u_\eps
- \div \left (\MA_\eps \grad u_\eps \right )
&=&
f
&
\text{ in } \Om,
\\
\grad u_\eps \cdot \bn - i k u_\eps
&=&
0
&
\text{ on } \partial \Om.
\end{array}
\right .
\end{equation*}
More rigorously, in weak form, we look for $u_\eps \in H^1(\Om)$ such that
\begin{equation}
\label{eq_helmholtz_eps_weak}
b_\eps(u_\eps,v) = (f,v)_\Om
\end{equation}
for all test functions $v \in H^1(\Om)$, where we introduced the sesquilinear form
\begin{equation*}
b_\eps(\phi,v)
\eq
-k^2(\mu_\eps \phi,v)_\Om-ik(\phi,v)_{\partial \Om}+(\MA_\eps\grad \phi,\grad v)_\Om
\end{equation*}
for any trial function $\phi \in H^1(\Om)$.

\subsection{Homogenized coefficients}
\label{section_homogenized_coefficients}

We now define homogenized versions of the oscillating
coefficients introduced above. For $\widehat{\mu}^\eps$, the homogenized
coefficient $\muH$ is the mean value of $\widehat{\mu}$ over
the periodic cell $Y$, i.e.,
\begin{equation}
\muH \eq \mean{\widehat{\mu}}.
\end{equation}

For $\widehat{\MA}^\eps$, the averaging process is more intricate,
and we follow standard homogenization theory~\cite{allaire_1992a,cioranescu_donato_1999a}.
In order to determine $\MAH$, we first introduce, for $1 \leq j \leq d$,
the auxiliary functions $\widehat{\chi}_j \in H^1_\sharp(Y) $ as the solution to
\begin{equation}
\label{eq_definition_hchi}
-\grad_{\by} \cdot (\widehat{\MA} \grad_{\by} \widehat{\chi}_j)
=
-
\sum_{\ell=1}^d
\frac{\partial}{\partial y_\ell}\widehat{\MA}_{j\ell},
\end{equation}
with zero mean value in $Y$. We will use the notation $\widehat{\bchi} \in \BH^1_\sharp(Y)$
for the vector with components $(\widehat{\chi}_1,\ldots,\widehat{\chi}_d)$.
We have $\widehat{\bchi} \in \BW_\sharp^{1,\infty}(Y)$ with
\begin{equation}
\label{eq_estimate_hchi}
\|\widehat{\bchi}\|_{\BW^{1,\infty}} \lesssim 1.
\end{equation}
This follows by standard elliptic regularity, and is detailed
in~\cite{chaumont-frelet_scattering_2021}. Similar arguments are
also given in more detail in the proof of Lemma~\ref{lemma_heta} below.

Let the matrix $\widehat{\MC}$ be defined by
\begin{equation}
(\widehat{\MC})_{j\ell} := \frac{\partial \widehat{\chi}_\ell}{\partial y_j}, 
\qquad 1 \le j,\ell \le d.
\end{equation}
With
\begin{equation}
\widehat{\MB} := \widehat{\MA}(\MI - \widehat{\MC}),
\end{equation}
the homogenized matrix coefficient $\MAH$ is the mean value of $\widehat{\MB}$
over the periodic cell $Y$, i.e.,
\begin{equation}
(\MAH)_{j\ell} = \mean{\widehat{\MB}_{j\ell}},
\qquad 1 \le j,\ell \le d.
\end{equation}

\subsection{Homogenized problem}

We can now consider a homogenized problem where the coefficients become
piecewise constants. Namely, we let
\begin{equation*}
\mu_0 \eq \left \{
\begin{array}{ll}
\muH & \text{ in } D,
\\
1 & \text{ in } \Omega \setminus D,
\end{array}
\right .
\qquad
\MA_0 \eq \left \{
\begin{array}{ll}
\MAH & \text{ in } D,
\\
\tens{I} & \text{ in } \Omega \setminus D.
\end{array}
\right .
\end{equation*}

Then, given $f \in L^2(\Om)$ and $k > 0$, the homogenized problem
consists in finding $u_0 \in H^1(\Om)$ such that
\begin{equation*}
\left \{
\begin{array}{rcll}
- k^2 \mu_0 u_0
- \div \left (\MA_0\grad u_0 \right )
&=&
f
& \text{in } \Om,
\\
\grad u_0 \cdot \bn - i k u_0
&=&
0
&
\text{on } \partial \Om.
\end{array}
\right .
\end{equation*}
Again, more rigorously, we understand this equation in a weak form,
whereby we demand that
\begin{equation}
\label{eq_helmholtz_homo_weak}
b_0(u_0,v) = (f,v)_{\Om}
\end{equation}
for all $v \in H^1(\Om)$, with $b_0$ defined like $b_\eps$
with $\mu_\eps$ and $\MA_\eps$ replaced by $\mu_0$ and $\MA_0$.

\subsection{Elementary stability results}

Since $D$ is smooth and the coefficients $\mu_0$ and $\MA_0$
are piecewise constant, the following theorem is a standard
consequence of the unique continuation principle (see
e.g.~\cite{ball_capdebosq_tseringxiao_2012a,wolff_1992a})
and (piecewise) elliptic regularity for transmission problem
(see e.g.~\cite{chaumont-frelet_wavenumber_2020,costabel2010corner}).
The lower bound on $\CstH$ can be shown
following~\cite{chaumont-frelet_wavenumber-explicit_2022}
using the assumption that $d(\partial \Om,D) \geq \hD$, see
also~\cite{galkowski2019optimal}.

\begin{proposition}
\label{proposition_bound_homo}
For all $f \in L^2(\Om)$, there exists a unique solution $u_0 \in H^1(\Om)$
to the homogenized problem in~\eqref{eq_helmholtz_homo_weak}. In addition,
we have
\begin{equation}
\label{eq_bound_homo_L2}
\CstH
\eq
\max_{\substack{f \in L^2(\Om) \\ \|f\|_{L^2(\Om)} = 1}}
k^2\|u_0\| < +\infty,
\end{equation}
and
\begin{equation}
\label{eq_bound_homo_H2}
k\|u_0\|_{H^1_k(\Om)} + |u_0|_{H^2(D)}
\lesssim
\CstH \|f\|_{L^2(\Om)}.
\end{equation}
Besides, the lower bound $\CstH \gtrsim 1$ holds true.
\end{proposition}

\begin{remark}
The behaviour of the constant $\CstH$ in terms of $k$ has been
thoroughly studied in the literature, and we present some facts
here. If the obstacle $D$ is star-shaped, $\mu_0 \leq 1$ and $\MA_0 \leq \MI$
in the sense of quadratic forms, then~$\CstH \sim k\hD$,
see~\cite{barucq_stability_2016,moiola_acoustic_2018}.
When these assumption are not satisfied, it can still be shown
that there exists $\beta > 0$ such that for all $\delta > 0$ solely depending on $d$,
there exists $K_\delta^\otimes \subset \R_+$ with $|K_\delta^\otimes| \leq \delta$
such that~$\CstH \lesssim_\delta (k\hD)^\beta$ whenever $k \notin K_\delta^\otimes$,
see~\cite{lafontaine_for_2021}. Here, the notation $\lesssim_\delta$ means
that in addition to the dependencies listed in Section~\ref{section_hidden_constants},
the hidden constant can also depend on $\delta$.
\end{remark}

For the highly heterogeneous scattering problem, $\eps$-dependant estimates
may be classically obtained as above. Indeed, for each $\eps$, the
coefficients $\mu_\eps$ and $\MA_\eps$ are still piecewise smooth.
The lower bound on $\Cste$ can also be obtained as in the homogenized
case.

\begin{proposition}
\label{proposition_bound_eps}
Let $0 < \varepsilon \leq \hD$ and $k > 0$. For all $f \in L^2(\Om)$,
there exists a unique $u_\eps \in H^1(\Om)$ solution to~\eqref{eq_helmholtz_eps_weak},
and we have
\begin{equation}
\label{eq_bound_eps_L2}
\Cste
\eq
k^2\sup_{\substack{f \in L^2(\Om) \\ \|f\|_{L^2(\Om)} = 1}}
\|u_\eps\|_{L^2(\Om)}
<
+\infty.
\end{equation}
In addition, if the right-hand side in~\eqref{eq_helmholtz_eps_weak}
is replaced by $\langle \psi, v \rangle$ for a generic $\psi \in (H^1(\Om))'$,
then there is still a unique solution $u_\eps \in H^1(\Om)$, we have
\begin{equation}
\label{eq_bound_eps_H1}
\|u_\eps\|_{H^1_k(\Om)} \lesssim \Cste \|\psi\|_{H^{-1}_k(\Om)}.
\end{equation}
We also have $\Cste \gtrsim 1$.
\end{proposition}

\begin{remark}
\label{remark_bound_layers}
In contrast to the homogenized problem, very little seems to be known
about the behaviour of $\Cste$ for generic values of $k$ and $\varepsilon$.
In fact, the only results we are aware of are the ones
from~\cite{chaumont-frelet_scattering_2021}. There, the authors
consider the case where (i) $\widehat{\mu}$ and $\widehat{\MA}$ do not
depend on one space variable (say, $\by_d$), (ii) the wave speed is larger
than one in $D$, and (iii) the obstacle satisfies the geometric condition
$\bx_d \bn_d \geq 0$ on $\Gamma$, where $\bn$ is the unit normal vector
pointing outside $D$. In this setting, the estimate
\begin{equation}
\label{eq_bound_cste_layer}
\Cste \lesssim (k\hD)^3
\end{equation}
has been established, uniformly for $0 < \eps \leq \hD$,
in~\cite[Theorem 1.5]{chaumont-frelet_scattering_2021}.
The estimate in~\eqref{eq_bound_cste_layer} additionally holds
under relaxed assumptions on $\widehat{\mu}$ and $\widehat{\MA}$
as compared to the one we made in Section~\ref{section_periodic_patterns},
allowing in particular for piecewise constant periodic patterns.
The theory in~\cite{chaumont-frelet_scattering_2021} is actually
performed with a Dirichlet-to-Neumann operator, as opposed
to an absorbing boundary condition, on $\partial \Omega$. However,
the results also apply to the present setting.
\end{remark}


\section{Uniform stability estimates for the highly heterogeneous scattering problem}
\label{section_uniform_stability}

The goal of this section is to establish that the constant $\Cste$
is uniformly bounded in $\eps$ for any fixed $k$. This result is
already known for the specific configuration described in Remark~\ref{remark_bound_layers}
above. However, to the best of our knowledge, it has not been rigorously
established in general. Our proof is based on arguments by contradiction
from~\cite{bouchitte_felbacq_2004a}.

\begin{lemma}[]
\label{lemma_continuity_Cste}
For any fixed $k > 0$, $\Cste$ is a continuous function of $\eps$.
\end{lemma}

\begin{proof}
Let $0 < \eps \leq \hD$ and consider $\eps' \in (\eps-\delta,\eps+\delta)$,
with $\delta > 0$ to be chosen sufficiently small. We are going to show that
\begin{equation*}
|\Cste - \Cstep| \lesssim (\Cste)^2\frac{|\eps-\eps'|}{\eps}\frac{\hD}{\eps},
\end{equation*}
which implies the continuity of $\Cste$ as a function of $\eps$.

Fix $f \in L^2(\Om)$ and denote by $u_\eps,u_{\eps'} \in H^1(\Om)$
the solutions to the highly heterogeneous scattering problem in~\eqref{eq_helmholtz_eps_weak},
with $\eps'$ instead of $\eps$ in the second case. Since $D$ is bounded and contains
the origin, we have
\begin{equation*}
\left |
\frac{\bx}{\eps}-\frac{\bx}{\eps'}
\right |
=
|\bx|\frac{|\eps'-\eps|}{\eps\eps'}
\leq
\hD\frac{|\eps'-\eps|}{\eps\eps'},
\end{equation*}
for all $\bx \in D$, which allows us to write that
\begin{align*}
b_\eps(u_\eps-u_{\eps'},v)
&=
b_{\eps'}(u_{\eps'},v)-b_{\eps}(u_{\eps',v})
\\
&=
-k^2 ((\widehat{\mu}^{\eps'}-\widehat{\mu}^{\eps})u_{\eps'},v)_D
+
((\widehat{\MA}^{\eps'}-\widehat{\MA}^{\eps})\grad u_{\eps'},\grad v)_D
\\
&\lesssim
\hD \frac{|\eps-\eps'|}{\eps\eps'}
\left (
|\widehat\mu|_{W^{1,\infty}(Y)}
+
|\widehat\MA|_{\MW^{1,\infty}(Y)}
\right )
\|u_{\eps'}\|_{H^1_k(\Om)}\|v\|_{H^1_k(\Om)}
\\
&\lesssim
\Cstep
\hD \frac{|\eps-\eps'|}{\eps\eps'}
\left (
|\widehat\mu|_{W^{1,\infty}(Y)}
+
|\widehat\MA|_{\MW^{1,\infty}(Y)}
\right )
\frac{\|f\|_{L^2(\Om)}}{k} \|v\|_{H^1_k(\Om)}.
\end{align*}

Using the bound in~\eqref{eq_bound_eps_H1}, it follows that
\begin{equation*}
k^2\|u_{\eps'}\|_{L^2(\Om)}
\leq
k^2\|u_\eps-u_{\eps'}\|_{L^2(\Om)}
+
k^2\|u_{\eps}\|_{L^2(\Om)}
\leq
C \hD \Cste\Cstep \frac{|\eps-\eps'|}{\eps\eps'}\|f\|_{L^2(\Om)}
+
\Cste\|f\|_{L^2(\Om)}
\end{equation*}
uniformly in $f$, from which we conclude that
\begin{equation*}
\Cstep-\Cste
\leq
C \hD \Cste\Cstep \frac{|\eps-\eps'|}{\eps\eps'}.
\end{equation*}
In this last two estimates and below, $C$ is a generic constant
with possible dependencies as described in Section~\ref{section_hidden_constants}.
Since this relation is symmetric, we have
\begin{equation*}
|\Cste-\Cstep|
\leq
C
\hD\Cste\Cstep\frac{|\eps-\eps'|}{\eps\eps'}
\leq
C\hD(\Cste)^2\frac{|\eps-\eps'|}{\eps\eps'}
+
C\hD\Cste\frac{|\eps-\eps'|}{\eps\eps'}|\Cste-\Cstep|,
\end{equation*}
so that
\begin{equation*}
\left (
1 - C\hD\Cste\frac{|\eps-\eps'|}{\eps\eps'}
\right )
|\Cste-\Cstep|
\leq
C\hD(\Cste)^2\frac{|\eps-\eps'|}{\eps\eps'},
\end{equation*}
leading to the desired bound if $\delta$ is sufficiently small This concludes the proof.
\end{proof}

\begin{remark}
In the proof of Lemma~\ref{lemma_continuity_Cste}, we used the
fact that $\widehat{\mu}$ and $\widehat{\MA}$ are Lipschitz
continuous. We believe, however, that this assumption can
be lightened. For instance, the bound
\begin{equation*}
|((\widehat{\MA}^{\eps'}-\widehat{\MA}^{\eps})\grad u_{\eps'},\grad v)_D|
\lesssim
\hD\frac{|\eps-\eps'|}{\eps\eps'}
|\widehat\MA|_{\MW^{1,\infty}(Y)}
|u_{\eps'}|_{H^1(\Om)}|v|_{H^1(\Om)}
\end{equation*}
could be replaced by
\begin{equation*}
|((\widehat{\MA}^{\eps'}-\widehat{\MA}^{\eps})\grad u_{\eps'},\grad v)_D|
\lesssim
\left (\hD\frac{|\eps-\eps'|}{\eps\eps'}\right )^s
|\widehat\MA|_{\MW^{s,\infty}(Y)}
|u_{\eps'}|_{H^{1+s'}(D)}|v|_{H^1(\Om)}
\end{equation*}
for some $s,s' > 0$, with the help of fractional Sobolev spaces.
In this case, we would need to show a bound of the form
\begin{equation*}
k |u_{\eps'}|_{H^{1+s'}(D)}
\lesssim
\Cstep \left (\frac{k}{\eps'}\right )^{-s'} \|f\|_{L^2(\Om)}
\end{equation*}
uniformly for $\eps' \in (\eps-\delta,\eps+\delta)$.
Whereas it feels likely that such bound holds true, actually establishing
it seems to lead to technicalities that we prefer avoiding here.
\end{remark}

\begin{theorem}[]
We have
\begin{equation*}
\Cst \eq \sup_{0 < \eps \leq \hD} \Cste < +\infty.
\end{equation*}
\end{theorem}

\begin{proof}
We argue by contradiction. In other words, we assume that for each
$n \in \N$, there exists $f_n \in L^2(\Om)$ with $\|f_n\|_{L^2(\Om)} = 1$
and $\eps_n \in (0,\hD]$ such that
\begin{equation}
\label{eq_contradiction}
k^2 \|u_{\eps_n}\|_{L^2(\Om)} \geq n = n\|f_n\|_{L^2(\Om)},
\end{equation}
where $u_{\eps_n}$ is the solution to~\eqref{eq_helmholtz_eps_weak}
with $\eps = \eps_n$ and right-hand side $f = f_n$.

Let us first observe that we necessarily have $\eps_n \to 0$ as $n \to +\infty$.
Indeed, if $\eps_n \geq \eps_\star > 0$ for all $n$, since $\Cste$ is continuous
on the closed interval $[\eps_\star,\hD]$, it is bounded and we have
\begin{equation*}
k^2\|u_{\eps_n}\|_{L^2(\Om)}
\leq
(\max_{\eps_\star \leq \eps \leq \hD} \Cste)
\|f_n\|_{L^2(\Om)}
=
\max_{\eps_\star \leq \eps \leq \hD} \Cste,
\end{equation*}
which contradicts~\eqref{eq_contradiction}.

We can therefore assume that $\eps_n \to 0$.
We then introduce the rescaled sequences
\begin{equation*}
\widehat{u}_n = \frac{1}{\|u_{\eps_n}\|_{L^2(\Om)}}u_{\eps_n},
\qquad
\widehat{f}_n = \frac{1}{\|u_{\eps_n}\|_{L^2(\Om)}} f_n.
\end{equation*}
We are going to show that a subsequence of $\widehat{u}_n$ converges
to $0$ in $L^2(\Om)$. This concludes the proof, since
$\|\widehat{u}_n\|_{L^2(\Om)} = 1$ for all $n \in \N$ by assumption.

Let us first point out that
\begin{equation}
\label{tmp_norm_f}
\|\widehat{f}_n\|_{L^2(\Om)}
\leq
\frac{\|f_n\|_{L^2(\Om)}}{\|u_{\eps_n}\|_{L^2(\Om)}}
\leq
\frac{k^2}{n}
\end{equation}
so that $\widehat{f}_n \to 0$ in $L^2(\Om)$. Besides, since
\begin{equation*}
\|\widehat{u}_n\|_{H^1_k(\Om)}^2
\lesssim
b(\widehat{u}_n,\widehat{u}_n) + k^2\|\widehat{u}_n\|_{L^2(\Om)}
=
(\widehat{f_n},\widehat{u}_n)_{\Om} + k^2\|\widehat{u}_n\|_{L^2(\Om)},
\end{equation*}
we also have that $\{\widehat{u}_n\}_{n \in \N}$ is bounded in $H^1(\Om)$.

Since $\{\widehat{u}_n\}_{n \in \N}$ is bounded in $H^1(\Om)$,
up to a subsequence that we still denote $\widehat{u}_n$,
it converges to some $u^0$ weakly in $H^1(\Om)$ and strongly in $L^2(\Om)$,
and $\grad \widehat{u}_n$ two-scale converges to
$\grad u^0 + \grad_y u^1$, for some $u^1 \in L^2(\Omega,H^1_\sharp(Y))$,
see~\cite[Proposition 1.14]{allaire_1992a}.
We are going to show that $u^0 = 0$, which concludes the proof.
Indeed, we should have $\|u^0\|_{L^2(\Om)} = \|\widehat{u}_n\|_{L^2(\Om)} = 1$
due to the strong convergence in $L^2(\Om)$.
Employing arguments based on the notion of two-scale
convergence~\cite{allaire_1992a,bouchitte_felbacq_2004a},
we can show that for any $v \in H^1(\Om)$
\begin{equation}
\label{tmp_two_scale}
\lim_{n \to +\infty} b_{\eps_n}(\widehat u_n,v) = b_0(u^0,v).
\end{equation}
More precisely, we exploit the two-scale convergence of $\grad \widehat{u}_n$
to show that $(u^0,u^1)$ solves a so-called two-scale problem. We can
then express $u^1$ in terms of $u^0$, leading to~\eqref{tmp_two_scale}
after averaging over the periodic cell $Y$. Since on the other hand,
\begin{equation*}
b_{\eps_n}(\widehat u_n,v) = (f_n,v)_\Om \to 0,
\end{equation*}
as $n \to +\infty$, we conclude that
\begin{equation*}
b_0(u^0,v) = 0
\end{equation*}
for all $v \in H^1(\Om)$. This implies that $u^0 = 0$
due to the well-posedness of the homogenized problem,
which finishes the proof.
\end{proof}

\section{Homogenization of the scattering problem}
\label{section_homogenization}

With the uniform bound on $\Cste$ established, we can continue
and provide homogenization error estimates for the scattering
problem.

\begin{theorem}[Frequency-explicit homogenization]
Let $f \in L^2(\Om)$ and consider the solutions
$u_\eps,u_0 \in H^1(\Om)$ of the associated heterogeneous
and homogenized problems in~\eqref{eq_helmholtz_eps_weak}
and~\eqref{eq_helmholtz_homo_weak}. Then, we have
\ifPREPRINT
\begin{multline}
\label{eq_homogenization_error}
k\|u_\eps-u_0-\eps\bchi_\eps \cdot \grad u_0\|_{H^1_k(\Om \setminus \Gamma)}
\lesssim
\Cst (k\eps)^{1/2}
\left (
k \|u\|_{H^1_k(\Om)} + |u|_{H^2(D)}
\right )
\\
\lesssim
\Cst \CstH (k\eps)^{1/2} \|f\|_{L^2(\Om)},
\end{multline}
\fi
where
\begin{equation*}
\bchi_\eps
\eq
\left \{
\begin{array}{ll}
\widehat{\bchi}^\eps & \text{ in } D
\\
\bo & \text{ in } \Om \setminus \overline{D}.
\end{array}
\right .
\end{equation*}
\end{theorem}

\begin{proof}
The second estimate in~\eqref{eq_homogenization_error} is an immediate consequence
of~\eqref{eq_bound_homo_H2}.

For the first estimate, we follow the proof given
in~\cite{chaumont-frelet_scattering_2021}.
In particular, we need to show that~\cite[Eq.~(4.26)]{chaumont-frelet_scattering_2021}
holds true for some
\begin{equation*}
\LM
\lesssim
\Cst \eps \left ( k\|u_0\|_{H^1_k(\Om)} + |u_0|_{H^2(D)} \right ).
\end{equation*}
We can do this following the original proof of~\cite{chaumont-frelet_scattering_2021}.
We first observe that the following variant of the estimate~\cite[Eq.~(4.28)]{chaumont-frelet_scattering_2021} holds true
\begin{equation*}
\|W_\varepsilon\|_{H^1_k(\Omega)} \lesssim \Cst \|\psi\|_{(H^1_k(\Omega))'}.
\end{equation*}
The next steps, namely~\cite[Lemmas 4.9 and 4.10]{chaumont-frelet_scattering_2021}
apply without modifications. We can then follow the proof
of~\cite[Lemma 4.7]{chaumont-frelet_scattering_2021}, where we observe
that all the occurrences of $\|u_0\|_{H^2(B_R)}$ can actually be replaced
by $|u_0|_{H^2(D)}$. Finally, we can conclude following the proof
of~\cite[Theorem 1.12]{chaumont-frelet_scattering_2021} where we can
replace any occurrence of $C_{\rm layer}(kR,kR_0,\widehat{n}_{\rm min})$
by $\Cst$.
\end{proof}

A direct consequence of the homogenization error estimate
in~\eqref{eq_homogenization_error} is the following comparison
of the stability constants.

\begin{corollary}
We have
\begin{equation}
\CstH \lesssim \Cst \lesssim
\left (
1
+
\Cst (k\eps)^{1/2}
\right )
\CstH.
\end{equation}
\end{corollary}

\section{The multiscale hybrid-mixed method}
\label{section_mhm}

In this section, we now describe the discretization method
we employ for the highly heterogeneous scattering problem
in~\eqref{eq_helmholtz_eps_weak}. Throughout this section,
we fix a polynomial degree $\ell \geq 0$.

\subsection{Curved mesh}
\label{section_curved_mesh}

We assume that $\Omega$ is covered by a (curved) simplicial mesh
$\CT_H$ satisfying standard smoothness assumptions, as stated,
e.g., in~\cite[Definition A.1]{chaumontfrelet_henning_2026a}.
We assume that for every (closed) element $K \in \CT_H$,
either $K \subset \overline{D}$ or $K \subset \overline{\Om} \setminus D$,
and we denote by $\CT_H^D \subset \CT_H$ the set of elements inside $D$.

The results below will be uniform in $\CK > 0$,
whenever (a) the mesh is $\ell$-regular
in the sense of~\cite[Eq. (55)]{chaumontfrelet_henning_2026a}
with constant $M_{\mathcal T_h}(\ell) \leq \CK$, and (b)
for all $K \in \CT_H^D$ and all $v \in H^1(K)$ such that
$\grad v \in \BH_0(\ddiv,K)$, we have $v \in H^2(K)$ with
\begin{equation}
\label{eq_assumption_mesh_H2}
|v|_{H^2(K)} \leq \CK \|\Delta v\|_{L^2(K)}.
\end{equation}
In other words, we demand that the element-wise Neumann problems for the
Laplacian admit a uniform ``$L^2(\Om)$ to $H^2(\Om)$'' regularity shift
on top of the standard mesh regularity assumption.

\begin{remark}
When the elements in the mesh are convex,~\eqref{eq_assumption_mesh_H2}
holds true with constant one, see~\cite[Chapter 3]{grisvard_1985a}.
This is for instance the case when straight simplices are employed.
For curved elements, the situation
is a bit more subtle. We first observe that if $D$ is itself convex,
the mesh can be designed in such a way that every $K \in \CT_H^D$
is convex, hence satisfying the assumption. If $D$ is not convex then it
cannot be meshed with convex elements. In this case, at least when $d=2$,
we can still get the regularity shift in~\eqref{eq_assumption_mesh_H2} under
the alternative assumption that all the edges of the triangles meet
at angle strictly less that $\pi$, which can also be reasonably
achieved in practice. We refer the reader to~\cite[Chapter 5.2]{grisvard_1985a}
for these regularity results on curved polygons. The authors expect that similar
comments might apply when $d=3$, although they are not aware of any rigorous proof.
\end{remark}

The diameter of each $K \in \CT_H$ is denoted by $H_K$, and we let
\begin{equation*}
H_{\rm max} \eq \max_{K \in \CT_H} H_K, \qquad H_{\rm min} \eq \min_{K \in \CT_H} H_K.
\end{equation*}
We also denote by $\bn_K$ the outward unit normal to $\partial K$,
and the set of faces of $\CT_H$ is denoted by $\CF_H$. Finally, we fix a reference
simplex $\widehat{K}$, and for each $K \in \CT_H$, we denote by
$\CF_K: \widehat{K} \to K$ the bijective map from $\widehat{K}$ to $K$
in~\cite[Definition A.1]{chaumontfrelet_henning_2026a}.
$\mathbb{J}_K$ is then the Jacobian matrix of $\CF_K$.

\begin{remark}
Our setting here assumes an exact meshing of the domain,
which might not be doable in practice. In contrast, isoparametric
elements (with polynomial maps $\CF_K$) are typically employed,
leading to a variational crime. In this case, it should be possible
to extend the present analysis by including a Strange-type lemma,
as is done in~\cite{chaumontfrelet_spence_2025a}. However, we avoid
these complications here.
\end{remark}
	
\subsection{Hybrid formulation}

The MHM method hinges on a hybrid formulation where the continuity
of the primal variable is relaxed. We thus consider the broken Sobolev space
\begin{equation*}
H^1(\CT_H)
\eq
\{ v \in L^2(\Om)\,:\, v|_K \in H^1(K)\ \forall K \in \CT_H \}.
\end{equation*}
In what follows, we extend the definition of the
inner products and sequilinear forms employed for $H^1(\Om)$
functions to functions in $H^1(\CT_H)$, whereby
the usual gradient is replaced by its element-wise
counterpart. Note that in particular, the sesquilinear
forms $b_\eps$ and $b_0$ can be naturally extended,
since functions of $H^1(\CT_H)$ admit $L^2(\partial \Om)$
traces. We, however, employ the notation $\|{\cdot}\|_{H^1_k(\CT_H)}$
instead of $\|{\cdot}\|_{H^1_k(\Om)}$ for function in $H^1(\CT_H)$
where the element-wise gradient is involved. We also use similar
notation for high-order Sobolev spaces. We also employ the same
notations with subscript $\CT_H^D$, where the support of integration
is restricted to $D$.

The continuity of the solution is enforced by the following
space of Lagrange multipliers
\begin{equation*}
H^{-1/2}(\partial \CT_H)
\eq
\left \{
\mu \in H^{-1/2}(\partial K)
\; | \;
\exists \bq \in \BH_0(\ddiv,\Omega);
\;
\mu|_{\partial K} = \bq \cdot \bn_K \quad \forall K \in \CT_H
\right \},
\end{equation*}
through the duality pairing 
\begin{equation*}
\langle \mu, v \rangle_{\partial \CT_H}
\eq
\sum_{K \in \CT_H} \langle \mu,v \rangle_{H^{-1/2}(\partial K),H^{1/2}(\partial K)}
\qquad
\forall (\mu,v) \in H^{-1/2}(\partial \CT_H) \times H^1(\CT_H).
\end{equation*}
We also introduce the dual norm
\begin{equation*}
\|\mu\|_{H^{-1/2}_k(\partial \CT_H)}
\eq
\max_{\substack{v \in H^1(\CT_H) \\ \|v\|_{H^1_k(\CT_H)} = 1}}
\langle \mu, v \rangle
\end{equation*}
for all $\mu \in H^{-1/2}(\partial \CT_H)$.

The primal hybrid formulation consists in finding a couple
$(\lambda_\eps,u_\eps) \in H^{-1/2}(\partial \CT_H) \times H^1(\CT_H)$ such that
\begin{equation}
\label{eq_hybrid_formulation}
\left\lbrace
\begin{array}{rcll}
b_{\eps}(u_\eps,v)
+
\ds\left\langle \lambda_\eps,v\right\rangle
&=&
(f,v)_{\Om}
&
\;\; \forall v \in H^1(\CT_H),
\\
\ds\left\langle \mu,u_\eps\right\rangle
&=&
0
&
\;\; \forall \mu \in H^{-1/2}(\partial \CT_H).
\end{array}
\right .
\end{equation}
This problem has a unique solution, and in fact, $u_\eps$
is the original solution to the highly heterogeneous
problem in~\eqref{eq_helmholtz_eps_weak}, whereas
$\lambda_\eps|_{\partial K} = (\MA_\eps\grad u_\eps) \cdot \bn_K$
for all $K \in \CT_H$. Similar comments naturally apply
to the homogenized problem. We refer the reader
to~\cite{chaumont-frelet_multiscale_2020,raviart_thomas_1977a}
for a proof of these facts.

\subsection{The MHM method}
\label{sec_mhm_formulation}

In the MHM formulation, we substitute $u_\eps$ for $\lambda_\eps$ in order to obtain
a problem set on the skeleton $\partial \CT_H$ of the mesh involving only
$\lambda_\eps$ as unknown. This is done by defining two operators
$T_\eps: H^{-1/2}(\partial \CT_H) \to H^1(\CT_H)$ and $\hT: L^2(\Omega) \to H^1(\CT_H)$
by requiring that
\begin{equation}
\label{eq_local_problems}
b_\eps(T_\eps \mu,v) = \langle \mu, v \rangle,
\qquad
b_\eps(\hT \phi,v) = (f,v)_\Om,
\end{equation}
for all $\mu \in H^{-1/2}(\partial \CT_H)$, $\phi \in L^2(\Om)$ and $v \in H^1(\CT_H)$.
These operators are actually defined locally, and evaluating them
amounts to solving element-wise Helmholtz problems with Neumann boundary
conditions. We also similarly define operators $T_0$ and $\hTo$,
with $b_\eps$ replaced by $b_0$.
These definitions are actually sound for sufficiently fine meshes.
This is established in~\cite{chaumont-frelet_multiscale_2020}, which we state below.

\begin{proposition}[Well-posedness of the local problems]
\label{proposition_well_poesdness}
There exists a constant $\tau$ solely depending on the shape-regularity
parameters of the mesh and the extremal values of $\widehat{\mu}$ and $\widehat{\MA}$
such that if $kH_{\max} \leq \tau$, then the operators $T_\eps$, $\hT$, $T_0$ and $\hTo$
are well-defined.In addition, we have
\begin{equation}
\label{eq_bound_local}
\|T_\eps \mu\|_{H^1_k(\Om)}
+
\|T_0 \mu\|_{H^1_k(\Om)}
\lesssim
\|\mu\|_{H^{-1/2}_k(\partial \CT_H)},
\qquad
\|\hT \phi\|_{H^1_k(\Om)}
+
\|\hTo \phi\|_{H^1_k(\Om)}
\lesssim
\frac{1}{k}\|\phi\|_{L^2(\Om)}
\end{equation}
for all $\mu \in H^{-1/2}(\partial \CT_H)$ and $\phi \in L^2(\Om)$.
\end{proposition}

A direct consequence is that we can indeed eliminate the unknown $u_\eps$
and reformulate~\eqref{eq_hybrid_formulation} in terms of $\lambda_\eps$. This
process can be though of as a Schur complement, and results in the following
skeletal variational formulation.  Assuming that $kH_{\max} \leq \tau$,
for all $f \in L^2(\Om)$, there exists a unique
$\lambda_\eps\in H^{-1/2}(\partial \CT_H)$ such that
\begin{equation}
\label{eq_mhm_eps}
\langle \mu,T_\eps\lambda_\eps\rangle_{\partial\CT_H}
= 
-\langle \mu,\hT f\rangle_{\partial\CT_H}    
\quad \forall \mu \in H^{-1/2}(\partial \CT_H).
\end{equation}
This $\lambda_\eps$ corresponds to the one in~\eqref{eq_hybrid_formulation},
and we have $u_\eps = T_\eps \lambda_\eps + \hT f$
for the solution of the original scattering problem.

Similarly, for the homogenized problem,
there exists a unique $\lambda_0 \in H^{-1/2}(\partial \CT_H)$ such that
\begin{equation}
\label{eq_mhm_homo}
\langle \mu,T_0\lambda_0\rangle_{\partial\CT_H}
= 
-\langle \mu,\hTo f\rangle_{\partial\CT_H}    
\quad \forall \mu \in H^{-1/2}(\partial \CT_H),
\end{equation}
and we have $u_0 = T_0 \lambda_0 + \hTo f$.

We are now ready to present the MHM method, which amounts
to selecting a discretization subspace $\Lambda_H \subset H^{-1/2}(\partial \CT_H)$.
Here, we focus on the space spanned by the normal traces of Raviart--Thomas elements.
Namely,
\begin{equation*}
\Lambda_H
\eq
\left \{
\mu_H \in H^{-1/2}(\partial \CT_H)
\; | \;
\exists \bq_H \in \BQ_h;
\quad
\mu_H|_{\partial K} = \bq_H \cdot \bn_K \quad \forall K \in \CT_H
\right \},
\end{equation*}
with
\begin{equation*}
\BQ_H
\eq
\left \{
\bq_H \in \BH_0(\ddiv,\Om)
\; | \;
((\operatorname{det}\mathbb{J}_K)^{-1}\mathbb{J}_K \bq_H)|_K \circ \CF_K^{-1} \in \BCP_\ell(\widehat{K}) + \bx \CP_\ell(\widehat{K})
\right \}.
\end{equation*}
The elements of $\Lambda_H$ are (mapped) polynomials single-valued on each face
$F \in \CF_H$, without compatibility condition across edge and vertices.
We refer the reader to~\cite[Chapters 9.2 and 14]{ern_guermond_2021a} for more a
extensive discussion of Raviart--Thomas elements and the associated Piola mapping
employed in the definition of $\BQ_H$.

The discrete problem then consists in finding $\lambda_H \in \Lambda_H$ such that
\begin{equation}
\label{eq_mhm_discrete}
\langle \mu_H, T_\eps\lambda_H\rangle_{\partial \CT_H}
=
-\langle \mu_H, \hT f\rangle_{\partial \CT_H}
\quad
\forall
\mu_H \in \Lambda_H.
\end{equation}
Once $\lambda_H$ is computed as the solution to \eqref{eq_mhm_discrete},
an approximation to $u_\eps$ is obtained by setting
\begin{equation}
\label{one_level_solution}
u_H \eq \hT f + T_\eps\lambda_H.
\end{equation}

\begin{remark}
Our definition of the MHM method in~\eqref{eq_mhm_discrete}
implies that we are able to evaluate the operators $T_\eps$
and $\hT$. This corresponds to solving (well-posed) local
Helmholtz problems in each mesh element, which in general
cannot be done analytically. In practice, a second-level
discretization scheme (by primal or mixed finite elements)
is therefore employed independently within each element
to approximately evaluate these operators. It leads to
a two-level method where the second level mesh size $h$
is typically chosen such that $h \sim \eps$. In this work,
we do not analyse the effect of the second-level discretization.
However, we refer the reader to~\cite{barrenechea2020multiscale}
for some results in this direction in the case of coercive PDEs.
\end{remark}

\section{Frequency-explicit convergence analysis}
\label{section_mhm_discretization}

\subsection{Abstract convergence theory}

We first provide an abstract stability and convergence analysis
for \eqref{eq_mhm_discrete}. Following the Schatz argument~\cite{schatz1974observation},
one can show that the method is quasi-optimal if the mesh is sufficiently
refined. The key quantity in this analysis is the so-called
the approximation factor. Specifically, we set
\begin{equation}
\label{eq_Capp}
\Capp \eq k
\max_{\substack{f \in L^2(\Om) \\ \|f\|_{L^2(\Om)} = 1}}
\min_{\mu_H \in \Lambda_H}
\|\lambda_\eps - \mu_H\|_{H^{-1/2}_k(\partial \CT_H)},
\end{equation}
where, for each $f \in L^2(\Om)$, $\lambda_\eps \in H^{-1/2}(\partial \CT_H)$
is the corresponding solution to~\eqref{eq_mhm_eps}.
The following result can be found for the MHM method
in~\cite{chaumont-frelet_multiscale_2020}.

\begin{proposition}[Asymptotic quasi-optimality]
\label{proposition_quasi_optimality}
There exists a constant $\LC^\star$ independent of $\eps$, $k$ and $H$
such that, if $\Capp \leq \LC^\star$, then there exists a unique
$\lambda_H \in \Lambda_H $ solution to \eqref{eq_mhm_discrete},
and we have
\begin{equation}
\label{error_inf}
\|\lambda_\eps - \lambda_H\|_{H^{-1/2}_k(\partial \CT_H)} 
\lesssim 
\min_{\mu_H \in \Lambda_H}
\|\lambda_\eps - \mu_H\|_{H^{-1/2}_k(\partial \CT_H)}.
\end{equation}
\end{proposition}

\subsection{Approximability estimates for the homogenized problem}

Employing the idea of regularity
splitting~\cite{chaumont-frelet_wavenumber_2020,melenk_wave_2011},
the following approximability estimate is derived
in~\cite{chaumont-frelet_multiscale_2020} for the homogenized problem.
			
\begin{proposition}[Approximability for the homogenized problem]
\label{lemma_interpolation}
For all $f \in L^2(\Om)$, there exists $\Pi_H\lambda_0 \in \Lambda_H$ such that
\begin{equation}
\label{eq_interpolation_RT}
k \|\lambda_0 - \Pi_H \lambda_0\|_{H^{-1/2}_k(\partial \CT_H)}
\lesssim 
\left ( kH + \CstH (kH)^{(\ell+1)} \right )  \|f\|_{L^2(\Om)},
\end{equation}
where $\lambda_0 \in H^{-1/2}(\partial \CT_H)$ is the solution
to the skeletal formulation of the homogenized problem in~\eqref{eq_mhm_homo}.
\end{proposition}

\subsection{Local homogenization results}

We also have the following local homogenization results.

\begin{lemma}[Local homogenization]
\label{lemma_local_homogenization}
Assume that $kH_{\max} \leq \tau$. Then for all $f \in L^2(\Om)$, we have
\begin{equation}
\label{eq_homogenization_hT}
k\|\hT f-\hTo f - \eps \bchi_\eps \cdot \grad (\hTo f)\|_{H^1_k(\CT_H)}
\lesssim
\left (
(k\eps)^{1/2}
+
\left (\frac{\eps}{H_{\min}}\right )^{1/2}
+
\frac{\eps}{H_{\min}}
\right )\|f\|_{L^2(\Om)}
\end{equation}
and
\begin{equation}
\label{eq_homogenization_T}
k
\|
T_\eps \lambda_0-T_0 \lambda_0 - \eps \bchi_\eps \cdot \grad (T_0 \lambda_0)
\|_{H^1_k(\CT_H)}
\lesssim
\CstH
\left (
(k\eps)^{1/2}
+
\left (\frac{\eps}{H_{\min}}\right )^{1/2}
+
\frac{\eps}{H_{\min}}
\right )\|f\|_{L^2(\Om)}.
\end{equation}
\end{lemma}

\begin{proof}
We first observe that the integrand in the left-hand sides
of~\eqref{eq_homogenization_hT} and~\eqref{eq_homogenization_T}
vanishes on all $K \in \CT_H$ such
that $K \subset \overline{\Om} \setminus D$. Therefore, we
only need to focus on the remaining elements $K \in \CT_H^D$.

We thus fix $K \in \CT_H^D$. Since $\grad(\hTo f) \cdot \bn = 0$ on $\partial K$,
the assumption we made on the mesh in~\eqref{eq_assumption_mesh_H2} ensures that
\begin{equation}
\label{tmp_hTf_H2}
|\hTo f|_{H^2(K)}
\lesssim
\|\Delta (\hTo f)\|_{L^2(K)}
\lesssim
k^2 \|\hTo f\|_{L^2(K)} + \|f\|_{L^2(K)}
\lesssim
\|f\|_{L^2(K)},
\end{equation}
where we used the estimate in~\eqref{eq_bound_local} in the last step.
Then~\eqref{eq_homogenization_hT} directly follows
the homogenization estimate in Theorem~\ref{theorem_local_homo}
below, after summing over the elements.

Similarly, the estimate in~\eqref{eq_homogenization_T},
follows from Theorem~\ref{theorem_local_homo}
if we can show that $w_0 \eq T_0\lambda_0 \in H^2(\CT_H^D)$
with $|w_0|_{H^2(\CT_H^D)} \lesssim \CstH\|f\|_{L^2(\Om)}$. To do so,
we observe that $w_0 = u_0-\hTo f$, where $u_0 \in H^2(\Om \setminus \Gamma)$
is the homogenized solution from~\eqref{eq_helmholtz_homo_weak}. It then follows that
\begin{equation*}
|w_0|_{H^2(\CT_H^D)}
\leq
|u_0|_{H^2(D)}
+
\|\hTo f\|_{H^2(\CT_H^D)}
\lesssim
\CstH \|f\|_{L^2(\Om)},
\end{equation*}
where we employed~\eqref{tmp_hTf_H2} above, the stability estimate for $u_0$
in~\eqref{eq_bound_homo_L2}, and the fact that $\CstH \gtrsim 1$.
\end{proof}

\begin{remark}
Here too, we use the assumption that $\widehat{\mu}$ and $\widehat{\MA}$
are Lipschitz functions. Indeed, we base our analysis on homogenization
results from~\cite{chaumont-frelet_scattering_2021} which requires such
regularity, at least when $d=3$. However, this assumption could be lightened
when $d=2$. Besides, the authors also believe that the analysis for $d=3$
could be performed under milder assumptions, provided that sharper arguments
are employed for the homogenization of the local problems.
\end{remark}

\subsection{Frequency-explicit convergence analysis}

We are now ready to establish our approximability estimate for the highly
heterogeneous problem.

\begin{theorem}
\label{theorem_Capp_eps}
If $kH_{\max} \leq \tau$, for all $f \in L^2(\Om)$, we have
\begin{multline}
\label{eq_estimate_mhm_interpolation_ms}
k\|\lambda_\eps-\pi_H\lambda_0\|_{H^{-1/2}_k(\partial \CT_H)}
\lesssim
\\
\left \{
kH_{\max}
+
\CstH
\left (
\left (\frac{\eps}{H_{\min}}\right )^{1/2}
+
\frac{\eps}{H_{\min}}
+
(kH_{\max})^{(\ell+1)}
\right )
+ 
\Cst\CstH (k\eps)^{1/2}
\right \}
\|f\|_{L^2(\Om)},
\end{multline}
where $\lambda_\eps,\lambda_0 \in H^{-1/2}(\partial \CT_H)$ are
to~\eqref{eq_mhm_eps} and~\eqref{eq_mhm_homo}.
In particular,
\begin{equation}
\label{eq_estimate_Capp_ms}
\Capp
\lesssim
kH_{\max}
+
\CstH
\left (
\left (\frac{\eps}{H_{\min}}\right )^{1/2}
+
\frac{\eps}{H_{\min}}
+
(kH_{\max})^{(\ell+1)}
\right )
+ 
\Cst\CstH (k\eps)^{1/2}.
\end{equation}
\end{theorem}

\begin{proof}
We start with the triangular inequality
\begin{equation*}
\|\lambda_\eps-\pi_H\lambda_0\|_{H^{-1/2}_k(\partial \CT_H)}
\leq
\|\lambda_\eps-\lambda_0\|_{H^{-1/2}_k(\partial \CT_H)}
+
\|\lambda_0-\pi_H\lambda_0\|_{H^{-1/2}_k(\partial \CT_H)}.
\end{equation*}
For the second term, we can readily apply the standard MHM interpolation estimate
in \eqref{eq_interpolation_RT}, leading to
\begin{equation*}
k\|\lambda_0-\pi_H\lambda_0\|_{H^{-1/2}_k(\partial \CT_H)}
\lesssim 
\left( kH + \Cst (kH)^{(\ell+1)}\right)  \|f\|_{0,\Om}.
\end{equation*}
For the first term, using~\cite[Corollary 3.4]{chaumont-frelet_multiscale_2020},
we write that
\begin{align*}
\|\lambda_\eps-\lambda_0\|_{H^{-1/2}_k(\partial \CT_H)}
&\lesssim
\|T_\eps \lambda_\eps-T_\eps\lambda_0\|_{H^1_k(\CT_H)}
\\
&\leq
\|T_\eps\lambda_\eps-T_0\lambda_0-\eps\bchi_\eps \cdot \grad(T_0\lambda_0)\|_{H^1_k(\CT_H)}
+
\|T_\eps\lambda_0-T_0\lambda_0-\eps\bchi_\eps \cdot \grad(T_0\lambda_0)\|_{H^1_k(\CT_H)}
\\
&\leq
\|u_\eps-u_0-\eps\bchi_\eps \cdot \grad u_0\|_{H^1_k(\CT_H)}
+
\|\hT f-\hTo f-\eps\bchi_\eps \cdot \grad(\hTo f)\|_{H^1_k(\CT_H)}
\\
&+
\|T_\eps\lambda_0-T_0\lambda_0-\eps\bchi_\eps \cdot \grad(T_0\lambda_0)\|_{H^1_k(\CT_H)}
\end{align*}
so that~\eqref{eq_estimate_mhm_interpolation_ms} follows
from~\eqref{eq_homogenization_error} and the estimates in
Lemma~\ref{lemma_local_homogenization}. Then,~\eqref{eq_estimate_Capp_ms}
is a direct consequence of the definition of $\Capp$ in \eqref{eq_Capp}.
\end{proof}

Combining Proposition~\ref{proposition_quasi_optimality}
together with Theorem~\ref{theorem_Capp_eps}, we can deliver
our final result. The last estimate in~\eqref{eq_mhm_error_u} follows
from~\cite[Corollary 3.4]{chaumont-frelet_multiscale_2020}.

\begin{corollary}[Error estimate for the MHM solution]
\label{corollary_error_estimate}
There exists $\zeta > 0$, independent of $\eps,k$ and $H_{\max}$
and $H_{\min}$, such that if
\begin{equation*}
\Cst\CstH (k\eps)^{1/2}
+
\CstH \left (
\left (\frac{\eps}{H_{\min}}\right )^{1/2}
+
\frac{\eps}{H_{\min}}
\right )
+
\CstH (kH)^{\ell+1} \leq \zeta,
\end{equation*}
then the MHM discretization of the highly heterogeneous problem in~\eqref{eq_mhm_discrete}
has a unique solution $\lambda_H \in \Lambda_H$, and
\begin{multline}
\label{eq_mhm_error_lambda}
\|\lambda_\eps-\lambda_H\|_{H^{-1/2}_k(\partial \CT_H)}
\lesssim
\min_{\mu_H \in \Lambda_H} \|\lambda-\mu_H\|_{H^{-1/2}_k(\partial \CT_H)}
\\
\lesssim
\left (
kH
+
\CstH (kH)^{\ell+1}
+
\CstH \left (
\left (\frac{\eps}{H_{\min}}\right )^{1/2}
+
\frac{\eps}{H_{\min}}
\right )
+
\Cst\CstH (k\eps)^{1/2}
\right ) \frac{\|f\|_{L^2(\Om)}}{k}.
\end{multline}
Finally, we have
\ifPREPRINT
\begin{multline}
\label{eq_mhm_error_u}
\|u_\eps-u_H\|_{H^1_k(\CT_H)}
\lesssim
\\
\left (
kH
+
\CstH (kH)^{\ell+1}
+
\CstH \left (
\left (\frac{\eps}{H_{\min}}\right )^{1/2}
+
\frac{\eps}{H_{\min}}
\right )
+
\Cst\CstH (k\eps)^{1/2}
\right ) \frac{\|f\|_{L^2(\Om)}}{k}.
\end{multline}
\fi
\end{corollary}

\begin{remark}
Our final estimates in Corollary~\ref{corollary_error_estimate}
seem to indicate that the method is not robust when $H_{\min} \lesssim \eps$.
However, in this case, we can actually employ the classical convergence
theory from~\cite{chaumont-frelet_multiscale_2020}, to replace terms
involving $\eps/H_{\min}$ by positive powers of $H_{\max}/\eps$, where
the negative power of $\eps$ comes from evaluating high-order (piecewise)
Sobolev norms the solution $u_\eps$.
\end{remark}

\begin{remark}
The division by $k$ of $\|f\|_{L^2(\Om)}$ in the right-hand sides
of~\eqref{eq_mhm_error_lambda} and~\eqref{eq_mhm_error_u} is natural.
Indeed, for the case of an incident plane wave where
$f = (k^2+\Delta)(\chi e^{ik\bd\cdot\bx})$, we do have $\|f\|_{L^2(\Om)} \sim k/\hD$.
\end{remark}

\section{Numerical examples}
\label{section_numerical_examples}

\subsection{Continuous setting}

We consider the case where $D = (-1,1)^2$ and $\Omega = (-2,2)^2$.
Note that, at least in principle, this violates our theoretical settings
where $D$ and $\Omega$ must be smooth. The periodic patterns are selected
as $\widehat{\mu} = 1$ and
\begin{equation*}
\widehat{\MA}(\bx) \eq 1 + 0.8\cos \left (\frac{2\pi}{\eps} \bx_1 \right ).
\end{equation*}
The right-hand side corresponds to the vertical incoming plane wave
$\xi = e^{ik\bx_2}$, so that we solve the scattering problem
with homogeneous right-hand side $f = 0$ and consider the inhomogeneous
boundary condition
\begin{equation*}
\grad u \cdot \bn - iku = g
\end{equation*}
on $\partial \Omega$, with $g = \grad \xi \cdot \bn - ik\xi$.

We note that the homogenized coefficient is given in this case by
\begin{equation*}
\MAH
\eq
\left (
\begin{array}{cc}
1 & 0
\\
0 & 3/5
\end{array}
\right ).
\end{equation*}
It is noteworthy that a horizontal incoming plane wave
$\xi = e^{ik\bx_1}$ would note see the obstacle $D$ in the limit
when $\varepsilon \to 0$, since the homogenized wave speed is uniformly
equal to one in the horizontal direction. We instead consider the
interesting case of a vertical plane wave for which the obstacle
is not transparent.

We focus on the frequencies $\omega = 2\pi$ and $10\pi$.
Representations of the solutions are given on Figure~\ref{figure_images}.
We also plot the (oscillating) vertical derivatives inside the scatterer
on Figure~\ref{figure_images_diff}

\begin{figure}
\input{figures/figure_images}
\caption{Representation of $\Re u$ in $\Omega$. The scatterer $D$ is represented by dotted lines.}
\label{figure_images}
\end{figure}
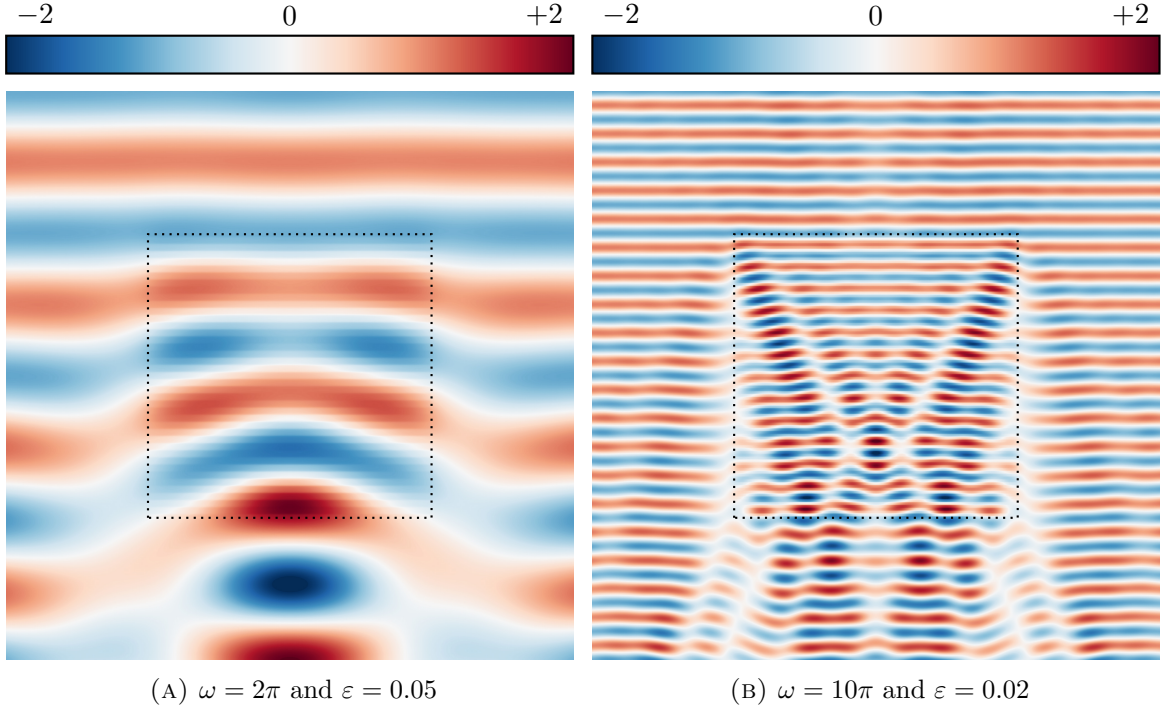

\begin{figure}
\input{figures/figure_images_diff}
\caption{Representation of $\Re \partial_2 u$ in $D$.}
\label{figure_images_diff}
\end{figure}
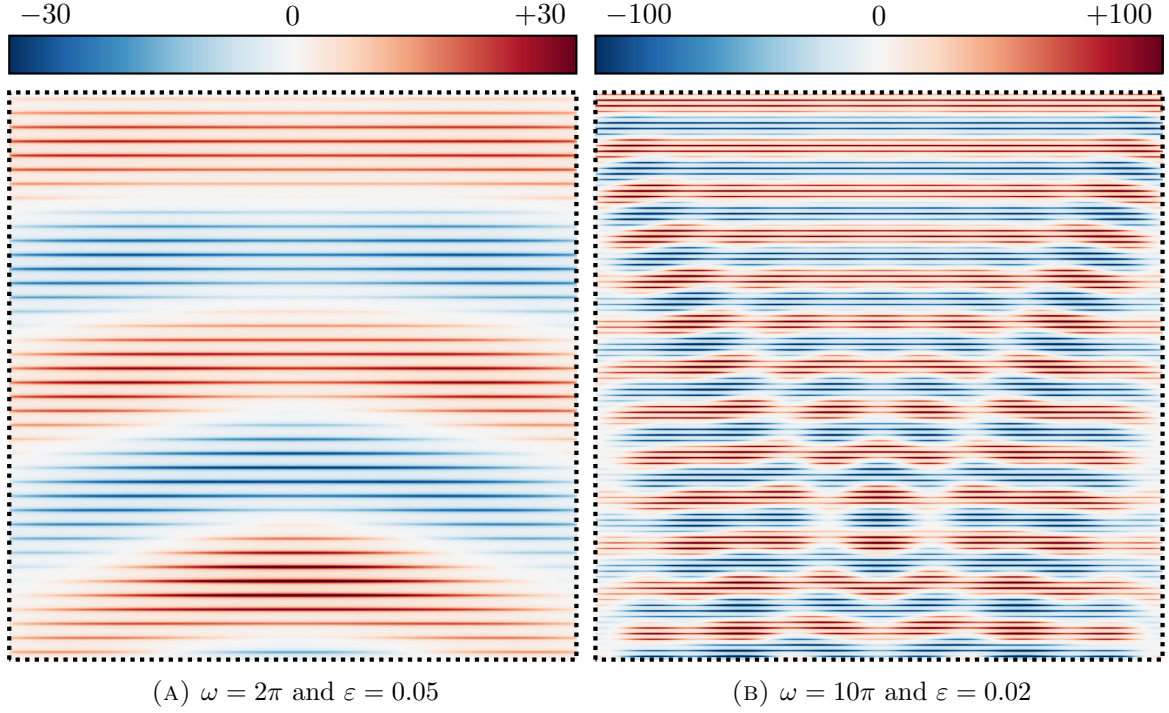

\subsection{Discrete setting}

All our discretizations are based on Cartesian grids.
More specifically, we employ an underlying fine grid
of $n \times n$ squares with $n \eq 8192$. The discretization
space is based on an $N \times N$ grid with $N$ ranging from
$16$ to $1024$. On this coarse grid, we employ the MHM method
described above with degree $\ell = 0$ and $1$. To compute the
basis functions, each cell of the $N \times N$ grid is divided
into $n/N \times n/N$ squares on which we employ Lagrange $\CQ_{\ell+1}$
elements. To assemble the matrix of the second-level finite element method,
we consider one constant value for $\MA_\eps$ per element by selecting the
value at the barycentre.  We note that this leads to the same approximation
of the coefficients for all values of $N$, since the underlying integration
grid for the second level is the same $n \times n$ fine grid.

To compute (approximate) errors, we consider a reference solution computed
as described as above with $N=2048$ and $\ell=1$. For comparison, all numerical
solutions are evaluated on a $2048 \times 2048$ Cartesian grid. More precisely,
we compute relative errors with
\begin{equation*}
\|u-u_H\|_{H^1_k(\Omega)}^2
\sim
\sum_{m,n=1}^{2048}
\left (
k^2|\widetilde u(\bx_{m,n})-u_H(\bx_{m,n})|^2
+
|\grad \widetilde u(\bx_{m,n})-\grad u_H(\bx_{m,n})|^2
\right ),
\end{equation*}
and
\begin{equation*}
\|u\|_{H^1_k(\Omega)}^2
\sim
\sum_{m,n=1}^{2048}
\left (
k^2|\widetilde u(\bx_{m,n})|^2
+
|\grad \widetilde u(\bx_{m,n})|^2
\right ),
\end{equation*}
where $\widetilde u$ is the reference solution and
$\bx_{m,n} = (-2+4(m-0.5)/2048,-2+4(n-0.5)/2048)$.

To highlight the benefit of employing a multiscale method,
we also run compute solutions with the same coarse grid,
but only one $Q_{\ell+1}$ element per coarse element (with
one value for $\MA_\eps$ taken at the barycentre). In this
case, the (two-level) MHM method becomes equivalent to
the primal-hybrid finite element method~\cite{bendali_2025a,raviart_thomas_1977a},
which is not designed to capture multiscale effects
and rely on standard polynomial basis functions.

\subsection{Results}

In all the figures below, we always represent the relative error
computed as described above. The relative error is plotted either
as a function of $N$ for a fixed $\eps$, or as a function of $\eps$
for a fixed $N$. Figures~\ref{figure_legend} summarizes the line styles
used.

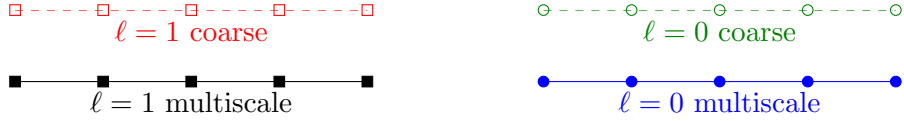
\begin{figure}
\input{figures/legend}
\caption{Line styles employed in all the other figures.}
\label{figure_legend}
\end{figure}

Figure~\ref{figure_F1_N} summarizes the results obtained for the lower
frequency $\omega = 2\pi$. We first note that the solid lines are always
below the dashed lines, meaning that there for all the considered values
of $\eps$, there is an interest in employing the multiscale method.
We also see here that the dashed lines are almost the same, meaning
that for the standard method, the limiting factor is really the approximation
of the coefficient, irrespectively of the chosen polynomial degree.
The fact that the multicale method is more performant even for ``large''
$\eps$ was already documented in the past~\cite{chaumont-frelet_multiscale_2020},
and is due to the fact that even if the coefficient is not multiscale,
the multiscale method still employs $k$-oscillating functions rather
than polynomials.

Still on Figure~\ref{figure_F1_N}, focusing now on the solid representing
the mutiscale method for $\ell=0$ and $1$. We see that for the largest
value of $\eps$, we are in the regime where ``$H \lesssim \eps$'', so that
no resonant effect occur. However, the resonant effect is present in all
the other cases, as can be seen by a localized error increase that happens
at larger values of $N$ as $\eps \to 0$. This is can be seen more clearly
on the left panel of Figure~\ref{figure_eps}, especially for the case $\ell=1$.
We see there that the error initially increases with $\eps$, before diminishing
linearly. We note that our theory only predicts a decrease as $\sqrt{\eps}$,
but the linear rate observed here could only be a pre-asymptotic effect.
We also see on Figure~\ref{figure_eps} that the error reaches a plateau
for $\ell = 0$. This is because the curves presented there are for a rather
small value of $N$, so that the error is dominated by terms involving $H$
rather than $\eps$.

We finally note that as can be seen on Figure~\ref{figure_F1_N}, although
the multiscale method with $\ell = 1$ is more performant, the gain is relatively 
small.

Figure~\ref{figure_F5_N} presents the results obtained in the case $\omega=10\pi$.
We can make some comments similar to the one we made for $\omega=2\pi$, namely that
it is always worse employing the multiscale method, even if the resonant effect is
present for the considered values of $H$ and $\eps$. The main difference with the
lower frequency case is that here, the higher-order method with $\ell=1$ is much
more performant that the one with $\ell=0$, in agreement with the analysis above.
In fact, the method with $\ell=0$ barely reaches $10\%$ of accuracy, whereas using
$\ell=1$ delivers a good accuracy in all cases in reasonable meshes. These comments
also apply in the right panel of Figure~\ref{figure_eps}, where we again observe a
linear convergence as $\eps \to 0$ after the resonant effect is past for $\ell=1$,
whereas a plateau is reached for $\ell=0$.

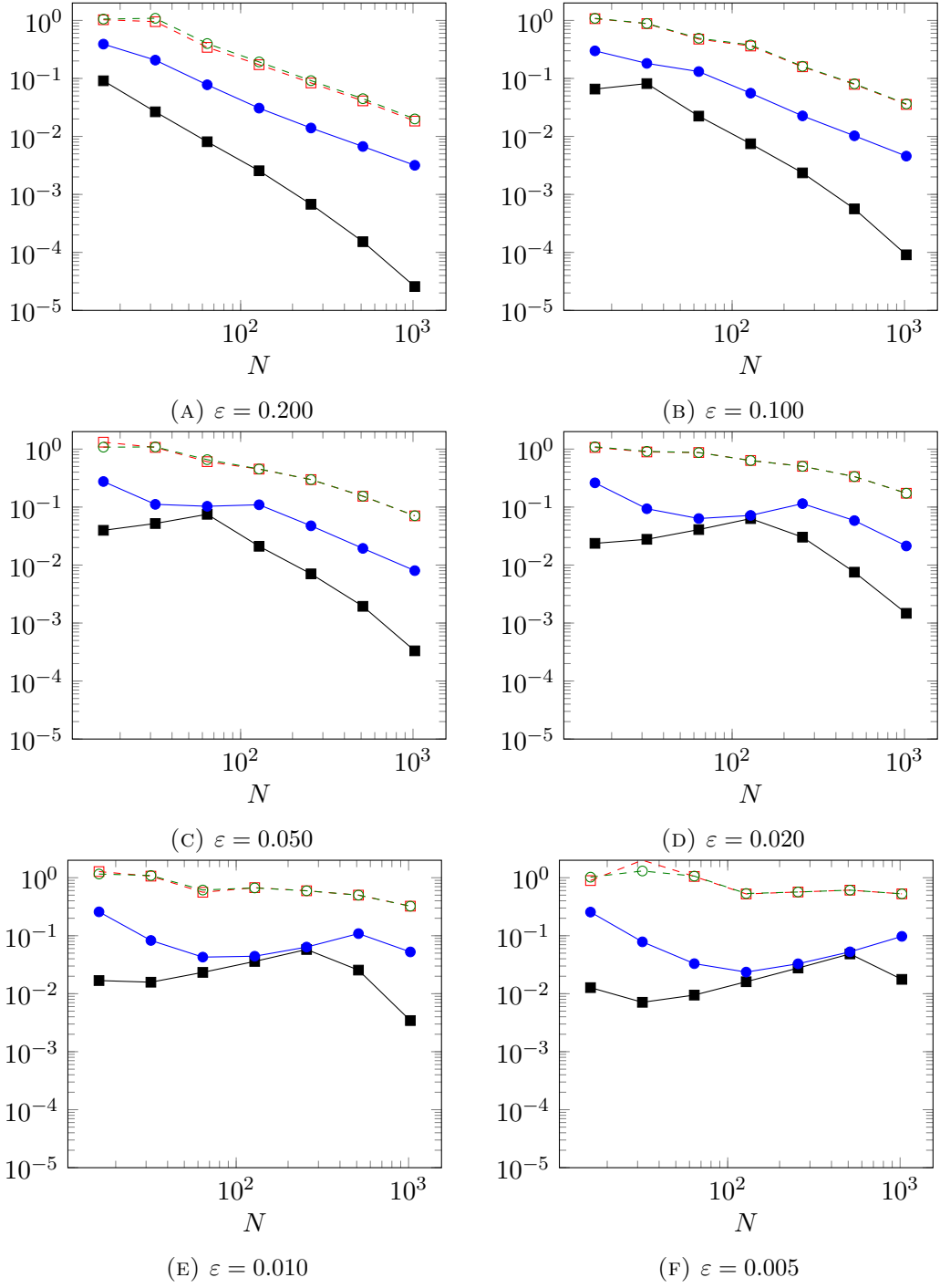
\begin{figure}
\input{figures/figure_F1_N}
\caption{Convergence as $N$ increase for $\omega = 2\pi$.}
\label{figure_F1_N}
\end{figure}

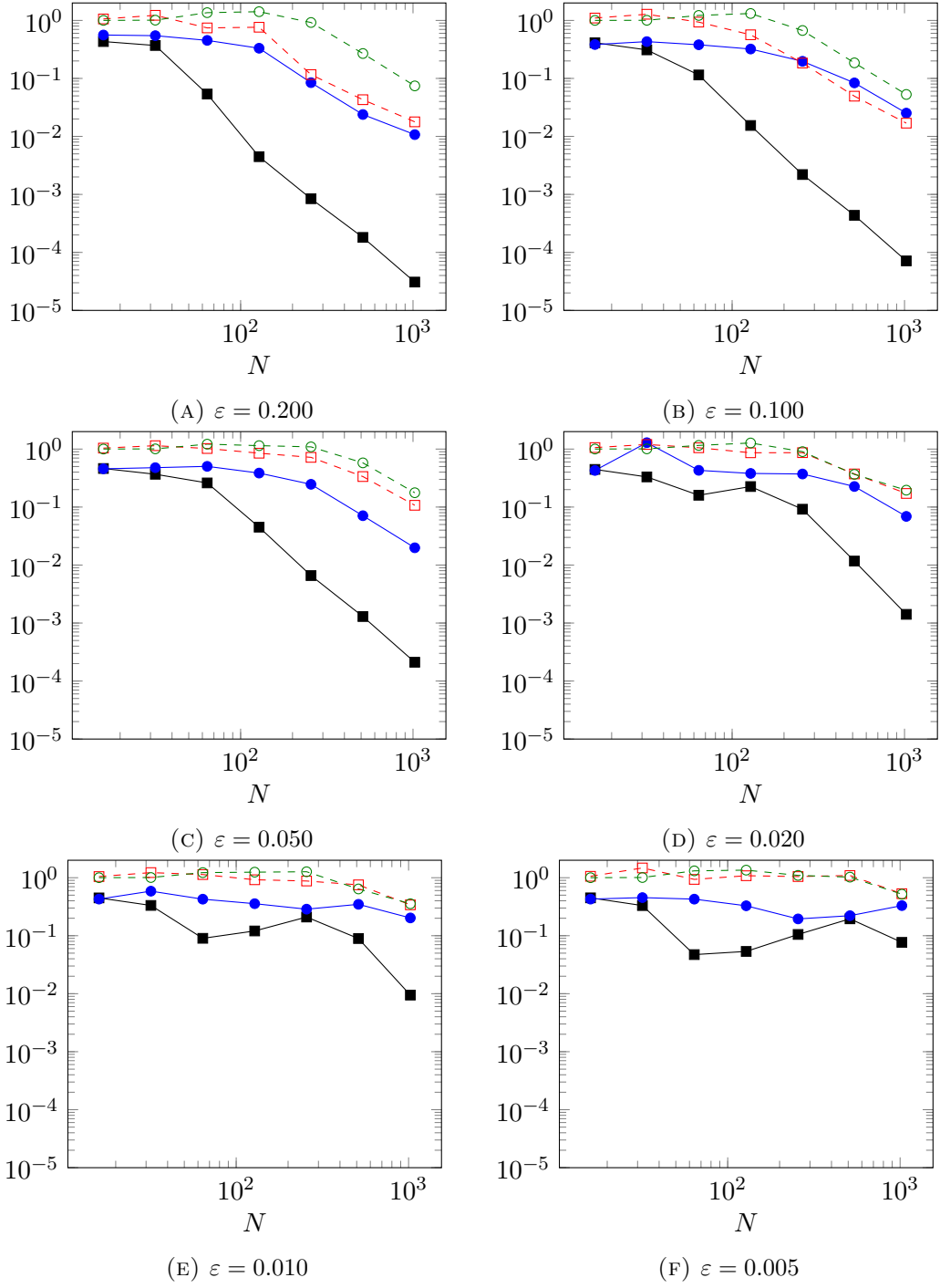
\begin{figure}
\input{figures/figure_F5_N}
\caption{Convergence as $N$ increase for $\omega = 10\pi$.}
\label{figure_F5_N}
\end{figure}

\begin{figure}
\input{figures/figure_eps}
\caption{Convergence as $\eps \to 0$.}
\label{figure_eps}
\end{figure}

\appendix

\section{Homogenization of local problems}
\label{appendix_homogenization}

This section is dedicated to proving local homogenization
results that we used in Lemma~\ref{lemma_local_homogenization} above.
Throughout this appendix, we assume that $kH_{\max} \leq \tau$ as in
Proposition~\ref{proposition_well_poesdness}.

We consider here a fixed $K \in \CT_H^D$, and we define the bilinear form
\begin{equation*}
b_\eps^K(\phi,v)
=
-k^2(\widehat{\mu}^\eps\phi,v)_K
+
(\widehat{\MA}^\eps \grad \phi,\grad v)_K
\end{equation*}
for all $\phi,v \in H^1(K)$, which coincides with $b_\eps$ for arguments
supported in $K$. We will also use
\begin{equation*}
b_0^K(\phi,v)
=
-k^2(\muH\phi,v)_K
+
(\MAH \grad \phi,\grad v)_K.
\end{equation*}

The following inf-sup condition follows from a Poincar\'e
inequality by exploiting the resoltion condition on the mesh.
It is established in~\cite{chaumont-frelet_multiscale_2020}.

\begin{proposition}
\label{proposition_inf_sup}
We have
\begin{equation*}
\inf_{\substack{\phi \in H^1(K) \\ \|\phi\|_{H^1_k(K)} = 1}}
\sup_{\substack{v \in H^1(K) \\ \|v\|_{H^1_k(K)} = 1}}
|b_\eps^K(\phi,v)|
\gtrsim
1.
\end{equation*}
\end{proposition}

A direct consequence of Proposition~\ref{proposition_inf_sup}
is that given $f \in L^2(K)$ and $\lambda \in H^{-1/2}(\partial K)$,
there exist unique $w_\eps, w_0 \in H^1(K)$ such that
\begin{equation}
\label{eq_appendix_problems}
b_\eps^K(w_\eps,v) = (f,v)_K + \langle \lambda, v \rangle_{\partial K},
\qquad
b_0^K(w_0,v) = (f,v)_K + \langle \lambda, v \rangle_{\partial K},
\end{equation}
for all $v \in H^1(K)$. We will establish the following
homogenization result.

\begin{theorem}
\label{theorem_local_homo}
Assume that $w_0 \in H^2(K)$ in~\eqref{eq_appendix_problems}. Then, we have
\ifPREPRINT
\begin{multline}
\label{eq_local_homogenization_error}
k\|w_\eps-w_0-\eps \bchi_\eps \cdot \grad w_0\|_{H^1_k(K)}
\lesssim
\\
\left (
(k\eps)^{1/2} + \left (\frac{\eps}{H_K}\right )^{1/2} + \frac{\eps}{H_K}
\right )
\left (
k\|w_0\|_{H^1_k(K)} + |w_0|_{H^2(K)}
\right ).
\end{multline}
\fi
\end{theorem}

\begin{proof}
Consider an arbitrary $v \in H^1(K)$. We have
\begin{equation*}
b_\eps^K(w_\eps,v) = (f,v)_K+\langle \lambda,v \rangle_{\partial K} = b_0^K(w_0,v),
\end{equation*}
and therefore
\begin{equation}
b_\eps^K(w_\eps-w_0-\eps \bchi^\eps \cdot \grad w_0,v)
=
b_0(w_0,v)-b_\eps(w_0,v)-b(\eps \bchi^\eps \cdot \grad w_0,v),
\end{equation}
which we regroup as
\ifPREPRINT
\begin{multline*}
|b_\eps^K(w_\eps-w_0-\eps \bchi^\eps \cdot \grad w_0,v)|
\leq
|k^2((\muH-\widehat{\mu}^\eps)w_0,v)_K|
\\
+
|
((\MAH-\widehat{\MA}^\eps)\grad u_0,\grad v)_K
-
b_\eps^K(\eps\widehat{\bchi}^\eps \cdot \grad w_0,v)
|.
\end{multline*}
\fi
We will prove in Lemmas~\ref{lemma_local_homo_L2} and~\ref{lemma_local_homo_H1}
below that
\begin{equation*}
k|k^2((\muH-\mu_\eps)w_0,v)_K|
\lesssim
\left (k\eps + \frac{\eps}{H_K}\right )k\|w_0\|_{H^1_k(K)} \|v\|_{H^1_k(K)}
\end{equation*}
and
\begin{multline*}
k|
((\MAH-\MA_\eps)\grad u_0,\grad v)_K
-
b_\eps^K(\eps\bchi^\eps \cdot \grad w_0,v)
|
\lesssim
\\
\left (
(k\eps)^{1/2}
+
\left (\frac{\eps}{H_K}\right )^{1/2}
\right )
\left (
k\|w_0\|_{H^1_k(K)} + |w_0|_{H^2(K)}
\right )
\|v\|_{H^1_k(K)},
\end{multline*}
this leads us to
\ifPREPRINT
\begin{multline*}
k|b_\eps^K(w_\eps-w_0-\eps \bchi^\eps \cdot \grad w_0,v)|
\lesssim
\\
\left \{
(k\eps)^{1/2}
+
\left (\frac{\eps}{H_K}\right )^{1/2}
+
\frac{\eps}{H_K}
\right )
\left (
k\|w_0\|_{H^1_k(K)} + |w_0|_{H^2(K)}
\right )
\|v\|_{H^1_k(K)},
\end{multline*}
\fi
and~\eqref{eq_local_homogenization_error} follows from the inf-sup condition
in Proposition~\ref{proposition_inf_sup}.
\end{proof}

The rest of this section is dedicated to establishing Lemmas~\ref{lemma_local_homo_L2}
and~\ref{lemma_local_homo_H1} to conclude the proof of Theorem~\ref{theorem_local_homo}.

\subsection{Correctors}
\label{subsection_correctors}

Following \cite[Section 7.1]{cioranescu_donato_1999a},
we need to introduce specific functions $D \times Y \to \R$
to perform the homogenization analysis of the local problems.
Such functions are often called ``correctors'' in this context.

\begin{lemma}[First-order corrector]
\label{lemma_hchi}
For any $w_0 \in H^2(K)$, let $w_1 \eq \widehat{\bchi} \cdot \grad w_0$.
Then, we have
\begin{equation*}
k\|w_1^\eps\|_{L^2(K)}
+
\|(\grad_{\bx} w_1)^\eps\|_{\BL^2(K)}
\lesssim
k \|w_0\|_{H^1_k(K)} + |w_0|_{H^2(K)}.
\end{equation*}
\end{lemma}

\begin{proof}
We have $w_1^\eps = \widehat{\bchi}^\eps \cdot \grad w_0$,
and it follows from~\eqref{eq_estimate_hchi} that
\begin{equation*}
\|w_1^\eps\|_{L^2(K)} \lesssim |w_0|_{H^1(K)} \leq \|w_0\|_{H^1_k(K)}.
\end{equation*}
For the other part of the estimate, we observe that
\begin{equation*}
|\grad_{\bx} u_1| \lesssim |\grad^2 w_0|
\end{equation*}
for a.e. $\bx \in K$, where $\grad^2 w_0$ is the Hessian of $w_0$. Indeed,
$\hchi$ does not depend on $\bx$, and it is
uniformly bounded. The conclusion then follows by integrating
this point-wise estimate.
\end{proof}

\begin{lemma}[Second-order corrector]
\label{lemma_heta}
Let $\heta$ be the unique element of $H^1_{\sharp}(Y)$ with
zero mean value such that
\begin{equation}
\label{eq_definition_heta}
-\grad_y \cdot \left ( \widehat{\MA} \grad_y \heta \right )
=
\widehat{\mu}-\muH.
\end{equation}
Then, we have  $\heta \in W_\sharp^{1,\infty}(Y)$ and
\begin{equation}
\label{eq_estimate_heta}
\|\heta\|_{W^{1,\infty}(Y)} \lesssim 1.
\end{equation}
\end{lemma}

\begin{proof}
Due to the compatibility condition
\begin{equation*}
\int_Y \widehat{\mu}-\muH = 0,
\end{equation*}
the problem defining $\widehat{\eta}$ is indeed well-posed.
Since $\widehat{\MA} \in \MW^{1,\infty}(K)$, we can rewrite
the definition of $\widehat{\eta}$ as
\begin{equation*}
-\Delta_{\by} \widehat{\eta}
=
\widehat{\mu}-\muH+\grad_y(\widehat{\MA}) \cdot \grad \widehat{\eta} =: \psi.
\end{equation*}
We already now that $\widehat{\eta} \in H^1_\sharp(Y)$, so that due
the assumption on $\widehat{\mu}$ and $\widehat{\MA}$, we have
$\psi \in H^1_\sharp(Y)$. It follows from standard regularity theory
that $\widehat{\eta} \in H^3_\sharp(Y)$, and we conclude that
$\widehat{\eta} \in W^{1,\infty}(Y)$ by a standard Sobolev embedding.
It is also clear that the constant only depends on the $W^{1,\infty}_\sharp(Y)$
norms of $\widehat{\MA}$ and $\widehat{\mu}$, justifying the estimate
in~\eqref{eq_estimate_heta}.
\end{proof}

The following trace inequality will be useful. It is an easy consequence
of the multiplicative trace inequality, see e.g.~\cite[Theorem 1.5.1.10]{grisvard_1985a}.
Here, the trace estimate holds uniformly under our regularity
assumptions on the mesh.

\begin{proposition}
For all $v \in H^1(K)$, we have
\begin{equation}
\label{eq_trace}
k\|v\|_{L^2(\partial K)}^2
\lesssim
\left (1+\frac{1}{kH_K}\right )\|v\|_{H^1_k(K)}^2.
\end{equation}
\end{proposition}

Finally, we will need a vector potential to deal with boundary conditions.
The result below is from~\cite[Theorem B.6]{chaumont-frelet_scattering_2021}.
The theorem in~\cite{chaumont-frelet_scattering_2021} states that the boundary
should be $C^{1,1}$, but this is in fact not required, and the regularity of
$\partial K$ suffices. The proof of this theorem requires that $\widehat{\MA}$
does not depend on $\by_d$, but in fact, this assumption
is only used in the proof of~\cite[Theorem B.3]{chaumont-frelet_scattering_2021}
to invoke a critical Sobolev inequality when the regularity of $\widehat{\MA}$
is low. Here, we assume that $\widehat{\MA} \in \MW^{1,\infty}(Y)$, so that
we can bypass this assumption. We finally point out that the
bound~\cite[Eq. (B.31)]{chaumont-frelet_scattering_2021} reads a bit
differently from the one below. This is because an assumption similar to
$kH_K \gtrsim 1$ is made in~\cite{chaumont-frelet_scattering_2021}. A careful
inspection of the proof of~\cite[Theorem B.5]{chaumont-frelet_scattering_2021}
shows that bound below holds true with the extra $(kH_K)^{-1}$ term.

\begin{proposition}[Existence of vector potential]\label{lemma_vector_potential}
There exists a vector potential $\bq : K \times Y \to \mathbb{C}^d$ such that
\begin{equation*}
\grad_{\by} \times \bq = (\widehat{\MB} - \MAH)\grad \widehat{w}_0.
\end{equation*}
In addition, we have
\begin{equation*}
\|(\grad_{\bx} \times \bq)^\eps\|_{L^2(K)} 
\lesssim
|w_0|_{H^2(K)},
\end{equation*}
and
\begin{equation*}
\|\grad \times \bq^\varepsilon \cdot \bn_K\|_{H^{-1/2}(\partial K)} 
\lesssim
\left (
1 + \frac{1}{kH_K}\frac{1}{k\eps} + \frac{1}{k\eps}
\right )^{1/2}
\left (
k \|w_0\|_{H^1_k(K)} + |w_0|_{H^2(K)}
\right ).
\end{equation*}
\end{proposition}

\subsection{End of the proof}

With the above results above in place, we are in position to conclude the proof of
Theorem~\ref{theorem_local_homo}. From here on, $w_0 \in H^2(K)$ is the homogenized
solution in~\eqref{eq_appendix_problems} and $v \in H^1(K)$ is an arbitrary test
function.

\begin{lemma}
\label{lemma_local_homo_L2}
We have
\begin{equation}
\label{eq_estimate_L2_homogenization}
\left |k^2\left ((\widehat{\mu}^\eps-\muH) w_0,v\right )_K\right |
\lesssim
\left (
k\eps + \frac{\eps}{H_K}
\right )
\|w_0\|_{H^1_k(K)}\|v\|_{H^1_k(K)}.
\end{equation}
\end{lemma}

\begin{proof}
Recalling the definition of $\heta$ from Lemma~\ref{lemma_heta}, we have
\begin{equation*}
\widehat{\mu}^\eps-\muH
=
(\widehat{\mu}-\muH)^\eps
=
- (\grad_{\by} \cdot (\widehat{\MA} \grad_{\by} \heta ))^\eps
=
-\eps
\grad \cdot 
(\widehat{\MA} \grad_{\by} \heta)^\eps,
\end{equation*}
and therefore
\begin{align*}
-\div (w_0 (\widehat{\MA} \grad_{\by} \heta))^\eps
&=
-w_0 \div (\widehat{\MA} \grad_{\by} \heta)^\eps
- \grad w_0 \cdot 
(\widehat{\MA} \grad_{\by} \heta)^\eps
\\
&=
-\frac{1}{\eps} (\widehat{\mu}_\eps-\muH)w_0
-\grad w_0 \cdot (\widehat{\MA} \grad_{\by} \heta)^\eps.
\end{align*}
Integrating by parts, we have
\begin{align*}
((\widehat{\mu}^\eps-\muH)w_0,v)_K
&=
-
\eps(\grad w_0 \cdot (\widehat{\MA} \grad_{\by} \heta)^\eps,v)_K
+
\eps 
(\div(w_0(\widehat{\MA}\grad_{\by}\heta))^\eps, v)_K
\\
&=	
-
\eps (\grad w_0 \cdot (\widehat{\MA} \grad_{\by} \heta)^\eps,v)_K
-
\eps (w_0 (\widehat{\MA} \grad_{\by} \heta)^\eps, \grad v)_K		
\\
&+
\eps 
(w_0 (\widehat{\MA} \grad_{\by} \heta )^\eps \cdot \bn_K,v)_{\partial K},
\end{align*}
leading to
\ifPREPRINT
\begin{multline*}
\left |
k^2\left ((\heps^\eps-n_0)w_0,v\right )_K
\right |
\lesssim
\\
k\eps
\left (
k\|w_0\|_{L^2(\partial K)} \|v\|_{L^2(\partial K)}
+
|w_0|_{H^1(K)}\|v\|_{L^2(K)}
+
k\|w_0\|_{L^2(K)} |v|_{H^1(K)}
\right ),
\end{multline*}
\fi
where we employed Lemma~\ref{lemma_heta} to estimate the terms related to $\heta$.
Then, \eqref{eq_estimate_L2_homogenization} follows since~\eqref{eq_trace} implies
that
\ifPREPRINT
\begin{multline*}
k\|w_0\|_{L^2(\partial K)} \|v\|_{L^2(\partial K)}
+
|w_0|_{H^1(K)}\|v\|_{L^2(K)}
+
k\|w_0\|_{L^2(K)} |v|_{H^1(K)}
\\
\lesssim
\left (
1 + \frac{1}{kH_K}
\right )
\|w_0\|_{H^1_k(K)}\|v\|_{H^1_k(K)}.
\end{multline*}
\fi
\end{proof}

\begin{lemma}
\label{lemma_local_homo_H1}
We have
\begin{multline}
\label{eq_estimate_Hx1_homogenization}
\left |
((\MAH-\widehat{\MA}^\eps) \grad w_0,\grad v)_K
-
b^K_\eps
\left (
\eps \widehat{\bchi}^\eps\cdot\grad w_0,v
\right )
\right |
\\
\lesssim
\eps
\left (1+ \frac{1}{k\varepsilon}\frac{1}{kH_K} + \frac{1}{k\varepsilon}\right )^{1/2}
\left (
k\|w_0\|_{H^1_k(K)} + |w_0|_{H^2(K)}
\right )
\|v\|_{H^1_k(K)}.
\end{multline}
\end{lemma}

\begin{proof}
In this proof, we will employ the properties of $\widehat{\bchi}$
described in Section~\ref{section_homogenized_coefficients},
and for shortness, we let $w_1 \eq \widehat{\bchi} \cdot \grad w_0$.
Using~\eqref{eq_delta_derivatives}, we write that
\begin{equation*}
\widehat{\MA}^\eps\grad (w_1^\eps) 
= 
\widehat{\MA}^\eps(\grad_{\bx} w_1)^\eps 
+
\frac{1}{\eps} 
\widehat{\MA}^\eps(\grad_{\by} w_1)^\eps 
= 
(\widehat{\MA}\grad_{\bx} w_1)^\eps 
-
\frac{1}{\eps}(\widehat{\MA}\widehat{\MC})^\eps\grad w_0,
\end{equation*}
As a result, we have
\begin{equation*}
\eps(\widehat{\MA}^\eps\grad (w_1^\eps) ,\grad v)_K
=
\eps((\widehat{\MA}\grad_{\bx} w_1)^\eps,\grad v)_K
-
((\widehat{\MA}\widehat{\MC})^\eps\grad w_0,\grad v)_K.
\end{equation*}
In addition, recalling that
$(\MAH- \widehat{\MA}^\eps) + (\widehat{\MA}\widehat{\MC})^\eps
=
\MAH - \widehat{\MB}^\eps$
and that
\begin{equation*}
\eps(\widehat{\MA}^\eps\grad (w_1^\eps),\grad v)_K
=
b_\eps^K (\eps w_1^\eps,v)
+
k^2 \eps (\widehat{\mu}^\eps  w_1^\eps,v)_K,
\end{equation*}
we arrive at
\ifPREPRINT
\begin{multline*}
((\MAH-\widehat{\MA}^\eps) \grad w_0,\grad v)_K
-
b^K_\eps(\eps w_1^\eps,v)
\\
=
((\MAH -\widehat{\MB}^\eps) \grad w_0,\grad v)_K
-
\eps((\widehat{\MA}\grad_{\bx} w_1)^\eps,\grad v)_K
+
k^2 \eps(\widehat{\mu}^\eps w_1^\eps,v)_K.
\end{multline*}
\fi

We now treat the first term in right-hand side
using the vector potential $\bq$. By~\eqref{eq_delta_derivatives}
\begin{equation*}
(\grad_{\by} \times \bq)^\eps
=
\eps (\grad \times(\bq^\eps) - (\grad_{\bx} \times \bq)^\eps)
\end{equation*}
thus, using the divergence theorem and the fact that
$\grad \cdot \grad  \times \bq^\eps = 0$,
\begin{align*}
((\MAH -\widehat{\MB}^\eps) \grad w_0,\grad v)_K
&=
((\grad_{\by} \times \bq)^\eps,\grad v)_K
\\
&=
\eps (\grad \times \bq ^\eps,\grad v)_K
-
\eps ((\grad_{\bx} \times  \bq )^\eps,\grad v)_K
\\
&=
\eps ((\grad \times \bq ^\eps)\cdot\bn_K,v)_{\partial K}
-
\eps ((\grad_{\bx} \times \bq )^\eps,\grad v)_K
\end{align*}
We then obtain
\begin{align*}
((\MAH-\widehat{\MA}^\eps) \grad w_0 ,\grad v)_K
-
b^K_\eps(\eps w_1^\eps, v)
=
&\eps
((\grad \times \bq^\eps)\cdot\bn_K, v)_{\partial K}
\\
&-
\eps
((\grad_{\bx} \times \bq)^\eps, \grad v)_K
-
\eps((\widehat{\MA}\grad_{\bx} w_1)^\eps,\grad v)_K
+
k^2 \eps(\widehat{\mu}^\eps w_1^\eps, v)_K.
\end{align*}
Now, it remains to bound the four terms in the right-hand side.
For the first, using Lemma~\ref{lemma_hchi}, we immediatly have
\begin{equation*}
k^2\eps |(\widehat{\mu}^\eps w_1^\eps,v)_K|
\lesssim
k^2\eps\|w_1^\eps\|_{L^2(K)} \|v\|_{L^2(K)}
\lesssim
\eps
k\|w_0\|_{H^1_k(K)}
\|v\|_{H^1_k(K)}.
\end{equation*}
For the second term, employing again Lemma~\ref{lemma_hchi},
we observe that
\begin{equation*}
\eps
|((\widehat{\MA}\grad_{\bx} w_1)^\eps,\grad v)_K|
\lesssim
\eps
\|(\grad_{\bx} w_1)^\eps\|_{\BL^2(K)}
|v|_{H^1(K)}
\lesssim
\eps|w_0|_{H^2(K)} \|v\|_{H^1_k(K)}.
\end{equation*}

Finally, the last two terms are easily dealt with using
Lemma~\ref{lemma_vector_potential}, since
\begin{equation*}
\eps |((\grad_{\bx} \times  \bq )^\eps,\grad v)_K|
\lesssim
\eps
\|(\grad_x \times \bq)^\varepsilon \|_{\BL^2(K)} 
|v|_{H^1(K)}
\lesssim
\eps
\left (
k\|w_0\|_{H^1_k(K)} + |w_0|_{H^2(K)}
\right )
\|v\|_{H^1_k(K)}
\end{equation*}
and 
\begin{align*}
\eps|((\grad \times \bq^\eps)\cdot\bn_K,v)_{\partial K}|
&\lesssim
\eps
\|\grad \times \bq^\varepsilon \cdot \bn_K\|_{H^{-1/2}(\partial K)} 
\|v\|_{H^{1/2}(K)}
\\
&\lesssim
\eps
\left (1+ \frac{1}{k\varepsilon}\frac{1}{kH_K} + \frac{1}{k\varepsilon}\right )^{1/2}
(k\|w_0\|_{H^1_k(K)} + |w_0|_{2,K})
\|v\|_{H^1_k(K)}.
\end{align*}
\end{proof}

\bibliography{bibliography}
\bibliographystyle{plain}

\end{document}

%% file: general_commands.tex
\renewcommand{\Re}{\operatorname{Re}}

\newcommand{\eq}{:=}

\newcommand{\ds}{\displaystyle}

\newcommand{\pd}[2]{\frac{\partial #1}{\partial #2}}

\newcommand{\grad}{\boldsymbol \nabla}
\renewcommand{\div}{\grad \cdot}

\newcommand{\ddiv}{\operatorname{div}}

\newcommand{\R}{\mathbb{R}}

\newcommand{\N}{\mathbb{N}}

\newcommand{\BH}{\boldsymbol H}

\newcommand{\BL}{\boldsymbol L}

\newcommand{\BQ}{\boldsymbol Q}

\newcommand{\BW}{\boldsymbol W}

\newcommand{\bd}{\boldsymbol d}

\newcommand{\bn}{\boldsymbol n}
\newcommand{\bo}{\boldsymbol o}

\newcommand{\bq}{\boldsymbol q}

\newcommand{\bx}{\boldsymbol x}
\newcommand{\by}{\boldsymbol y}

\newcommand{\CF}{\mathcal F}

\newcommand{\CK}{\mathcal K}

\newcommand{\CP}{\mathcal P}
\newcommand{\CQ}{\mathcal Q}

\newcommand{\CT}{\mathcal T}

\newcommand{\LC}{\mathscr C}

\newcommand{\LM}{\mathscr M}

\newcommand{\BCP}{\boldsymbol{\CP}}



%% file: specific_commands.tex
\newcommand{\Om}{\Omega}

\newcommand{\Cst}{\LC_{\rm st}}

\newcommand{\Capp}{\LC_{\rm app}}

\newtheorem{theorem}{Theorem}
\newtheorem{lemma}[theorem]{Lemma}
\newtheorem{corollary}[theorem]{Corollary}
\newtheorem{proposition}[theorem]{Proposition}
\newtheorem{remark}[theorem]{Remark}

\numberwithin{equation}{section}
\numberwithin{theorem}{section}

\newcommand{\eps}{\varepsilon}

\newcommand{\Cste}{\LC_{\rm st}^\eps}
\newcommand{\Cstep}{\LC_{\rm st}^{\eps'}}

\newcommand{\hchi}{\widehat \chi}

\newcommand{\heta}{\widehat \eta}

\newcommand{\heps}{\widehat \eps}

\newcommand{\mean}[1]{\langle #1 \rangle_Y}

\newcommand{\hT}{\widehat{T}_\eps}
\newcommand{\hTo}{\widehat{T}_0}

\newcommand{\CstH}{\Cst^{\rm H}}

\newcommand{\bxi}{\boldsymbol \xi}

\newcommand{\Amin}{A_{\rm min}}
\newcommand{\Amax}{A_{\rm max}}

\newcommand{\hD}{\ell_D}

\newcommand{\tens}[1]{\underline{\boldsymbol #1}}

\newcommand{\MW}{\tens{W}}

\newcommand{\MA}{\tens{A}}
\newcommand{\MI}{\tens{I}}
\newcommand{\MB}{\tens{B}}
\newcommand{\MC}{\tens{C}}

\newcommand{\MAH}{\tens{A}_{\rm H}}
\newcommand{\muH}{\mu_{\rm H}}

\newcommand{\bchi}{\boldsymbol \chi}

%% file: slope_triangle.tex
\newcommand{\SlopeTriangle}[6]
{

    \pgfplotsextra
    {
        \pgfkeysgetvalue{/pgfplots/xmin}{\xmin}
        \pgfkeysgetvalue{/pgfplots/xmax}{\xmax}
        \pgfkeysgetvalue{/pgfplots/ymin}{\ymin}
        \pgfkeysgetvalue{/pgfplots/ymax}{\ymax}

        \pgfmathsetmacro{\xArel}{#1}
        \pgfmathsetmacro{\yArel}{#3}
        \pgfmathsetmacro{\xBrel}{#1-#2}
        \pgfmathsetmacro{\yBrel}{\yArel}
        \pgfmathsetmacro{\xCrel}{\xArel}

        \pgfmathsetmacro{\lnxB}{\xmin*(1-(#1-#2))+\xmax*(#1-#2)} 
        \pgfmathsetmacro{\lnxA}{\xmin*(1-#1)+\xmax*#1} 
        \pgfmathsetmacro{\lnyA}{\ymin*(1-#3)+\ymax*#3} 
        \pgfmathsetmacro{\lnyC}{\lnyA+#4*(\lnxA-\lnxB)}
        \pgfmathsetmacro{\yCrel}{\lnyC-\ymin)/(\ymax-\ymin)} 

        \coordinate (A) at (rel axis cs:\xArel,\yArel);
        \coordinate (B) at (rel axis cs:\xBrel,\yBrel);
        \coordinate (C) at (rel axis cs:\xCrel,\yCrel);

        \draw[#6]   (A)-- node[anchor=north] {#5}
                    (B)--
                    (C)--
                    cycle;
    }
}

%% file: figures/figure_images.tex
\center

\begin{minipage}{.49\linewidth}
\begin{tikzpicture}
\draw (-3.75,4.00) node[anchor=south west,inner sep=0] {\includegraphics[width=7.5cm,height=0.5cm]{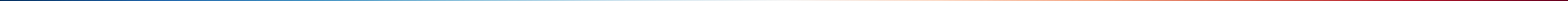}};
\draw[thick] (-3.75,4.00) rectangle (3.75,4.50);

\draw (-3.75,4.50) node[anchor=south west] {$-2$};
\draw ( 0.00,4.50) node[anchor=south] {$ 0$};
\draw ( 3.75,4.50) node[anchor=south east] {$+2$};

\draw (-3.75,-3.75) node[anchor=south west,inner sep=0pt] {\includegraphics[width=7.5cm]{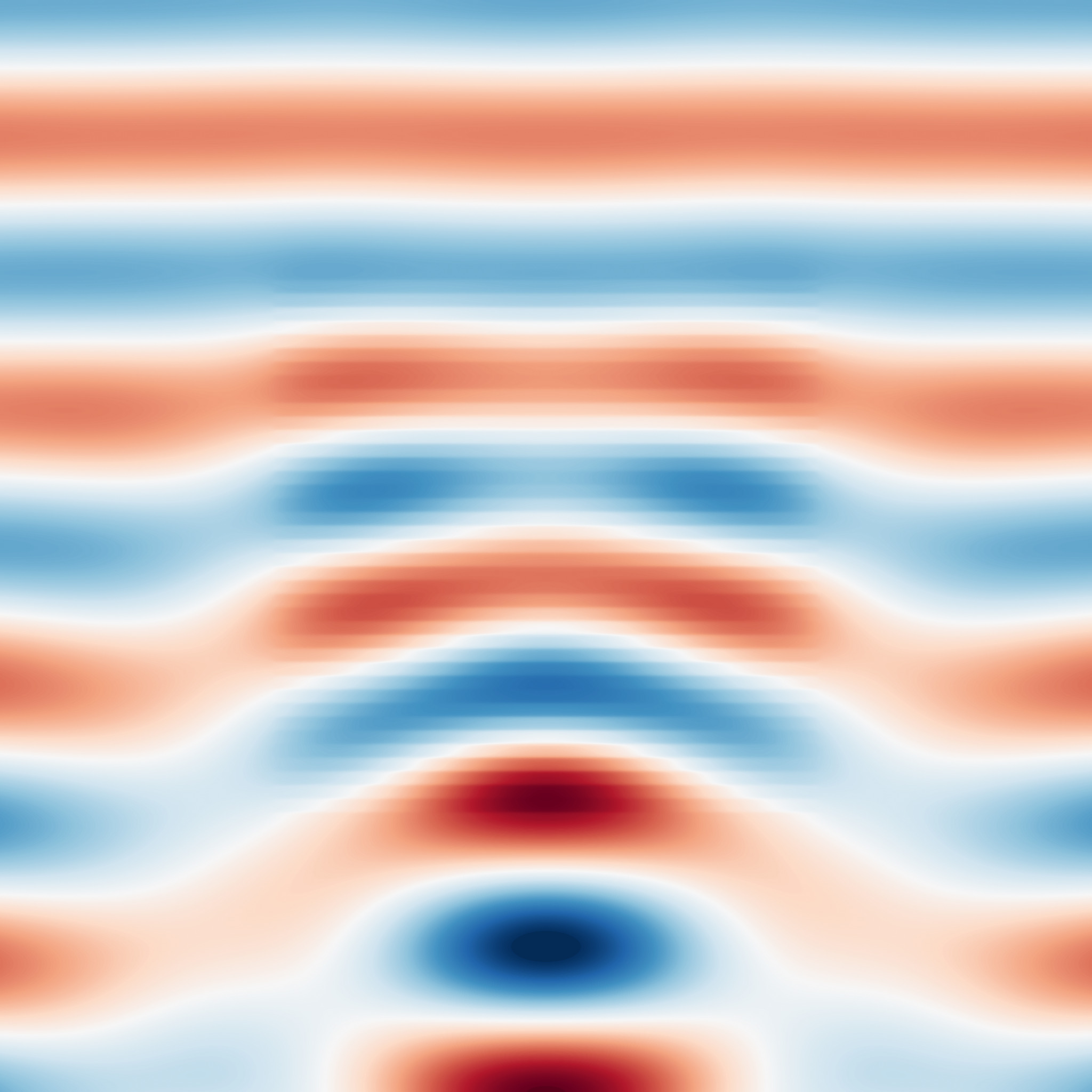}};

\draw[thick,dotted] (-1.875,-1.875) rectangle (1.875,1.875);
\end{tikzpicture}
\subcaption{$\omega = 2\pi$ and $\eps = 0.05$}
\end{minipage}
\begin{minipage}{.49\linewidth}
\begin{tikzpicture}
\draw (-3.75,4.00) node[anchor=south west,inner sep=0] {\includegraphics[width=7.5cm,height=0.5cm]{figures/data/bwr}};
\draw[thick] (-3.75,4.00) rectangle (3.75,4.50);

\draw (-3.75,4.50) node[anchor=south west] {$-2$};
\draw ( 0.00,4.50) node[anchor=south] {$ 0$};
\draw ( 3.75,4.50) node[anchor=south east] {$+2$};

\draw (-3.75,-3.75) node[anchor=south west,inner sep=0pt] {\includegraphics[width=7.5cm]{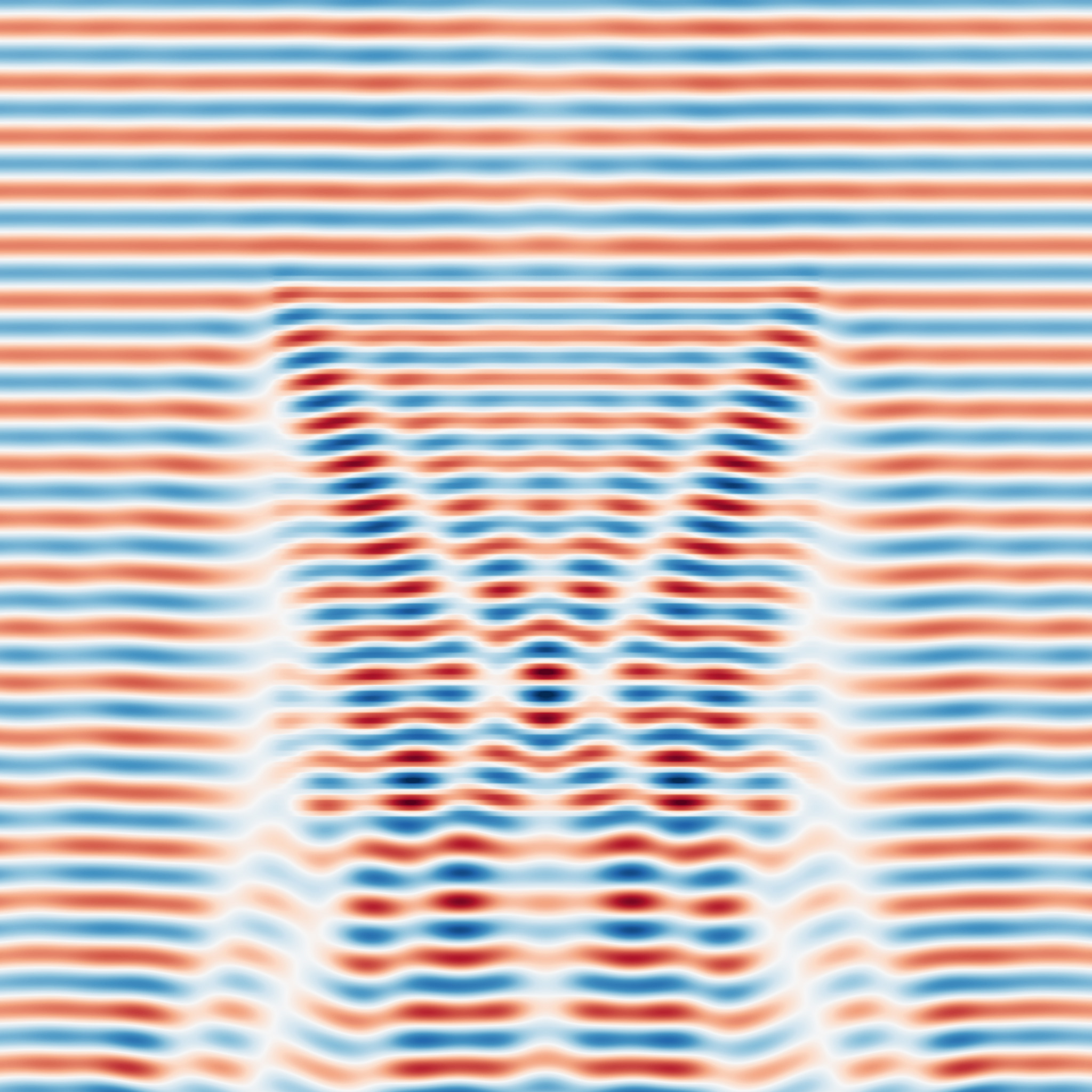}};

\draw[thick,dotted] (-1.875,-1.875) rectangle (1.875,1.875);
\end{tikzpicture}
\subcaption{$\omega = 10\pi$ and $\eps = 0.02$}
\end{minipage}

%% file: figures/figure_images_diff.tex
\center

\begin{minipage}{.49\linewidth}

\begin{tikzpicture}
\draw (-3.75,4.00) node[anchor=south west,inner sep=0] {\includegraphics[width=7.5cm,height=0.5cm]{figures/data/bwr}};
\draw[thick] (-3.75,4.00) rectangle (3.75,4.50);

\draw (-3.75,4.50) node[anchor=south west] {$-30$};
\draw ( 0.00,4.50) node[anchor=south] {$ 0$};
\draw ( 3.75,4.50) node[anchor=south east] {$+30$};

\draw[ultra thick,dotted] (-3.75,-3.75) rectangle (3.75,3.75);

\clip (-3.75,-3.75) rectangle (3.75,3.75);
\draw (-7.50,-7.50) node[anchor=south west,inner sep=0pt] {\includegraphics[width=15cm]{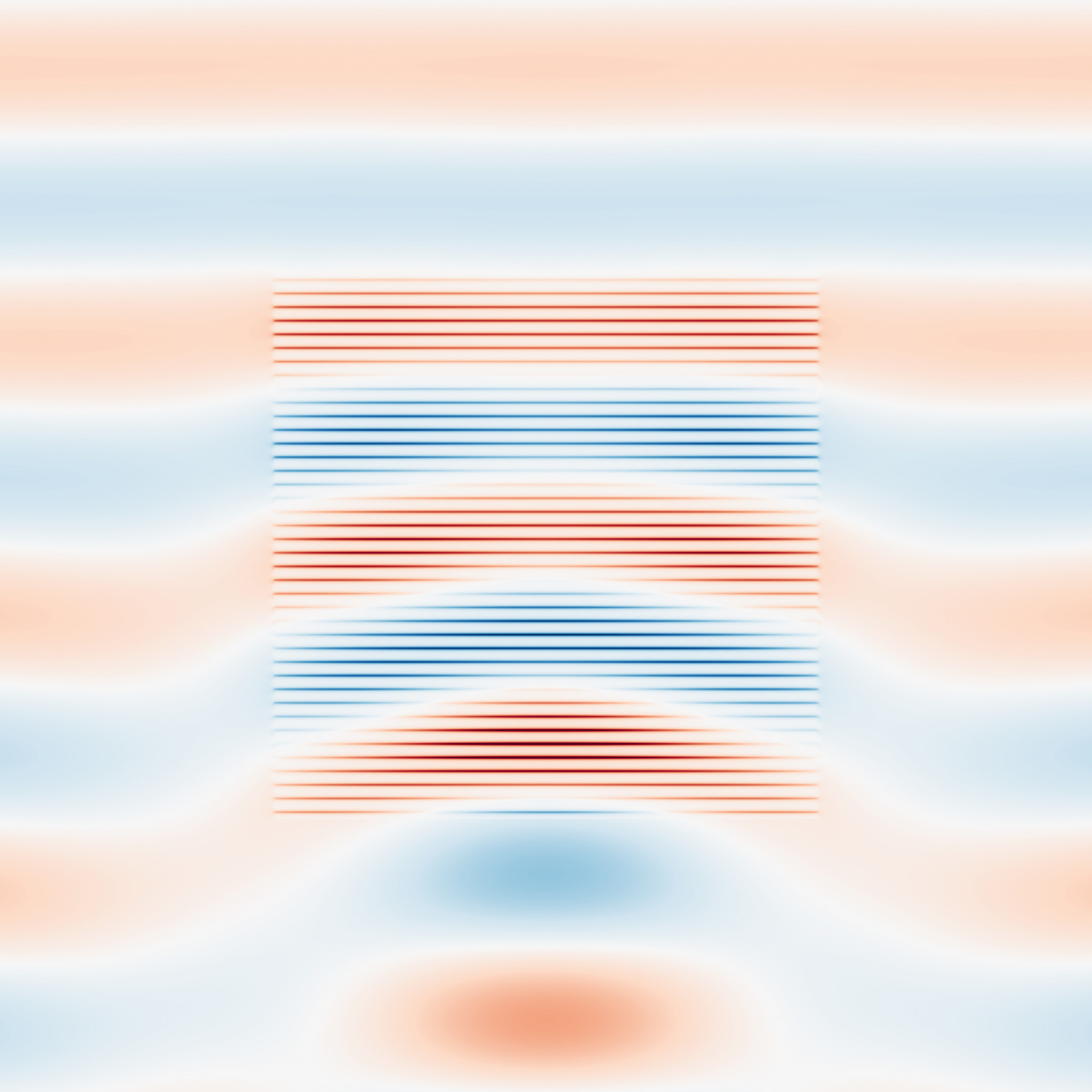}};

\draw[ultra thick,dotted] (-3.75,-3.75) rectangle (3.75,3.75);
\end{tikzpicture}
\subcaption{$\omega = 2\pi$ and $\eps = 0.05$}
\end{minipage}
\begin{minipage}{.49\linewidth}
\begin{tikzpicture}
\draw (-3.75,4.00) node[anchor=south west,inner sep=0] {\includegraphics[width=7.5cm,height=0.5cm]{figures/data/bwr}};
\draw[thick] (-3.75,4.00) rectangle (3.75,4.50);

\draw (-3.75,4.50) node[anchor=south west] {$-100$};
\draw ( 0.00,4.50) node[anchor=south] {$ 0$};
\draw ( 3.75,4.50) node[anchor=south east] {$+100$};

\draw[ultra thick,dotted] (-3.75,-3.75) rectangle (3.75,3.75);

\clip (-3.75,-3.75) rectangle (3.75,3.75);
\draw (-7.50,-7.50) node[anchor=south west,inner sep=0pt] {\includegraphics[width=15cm]{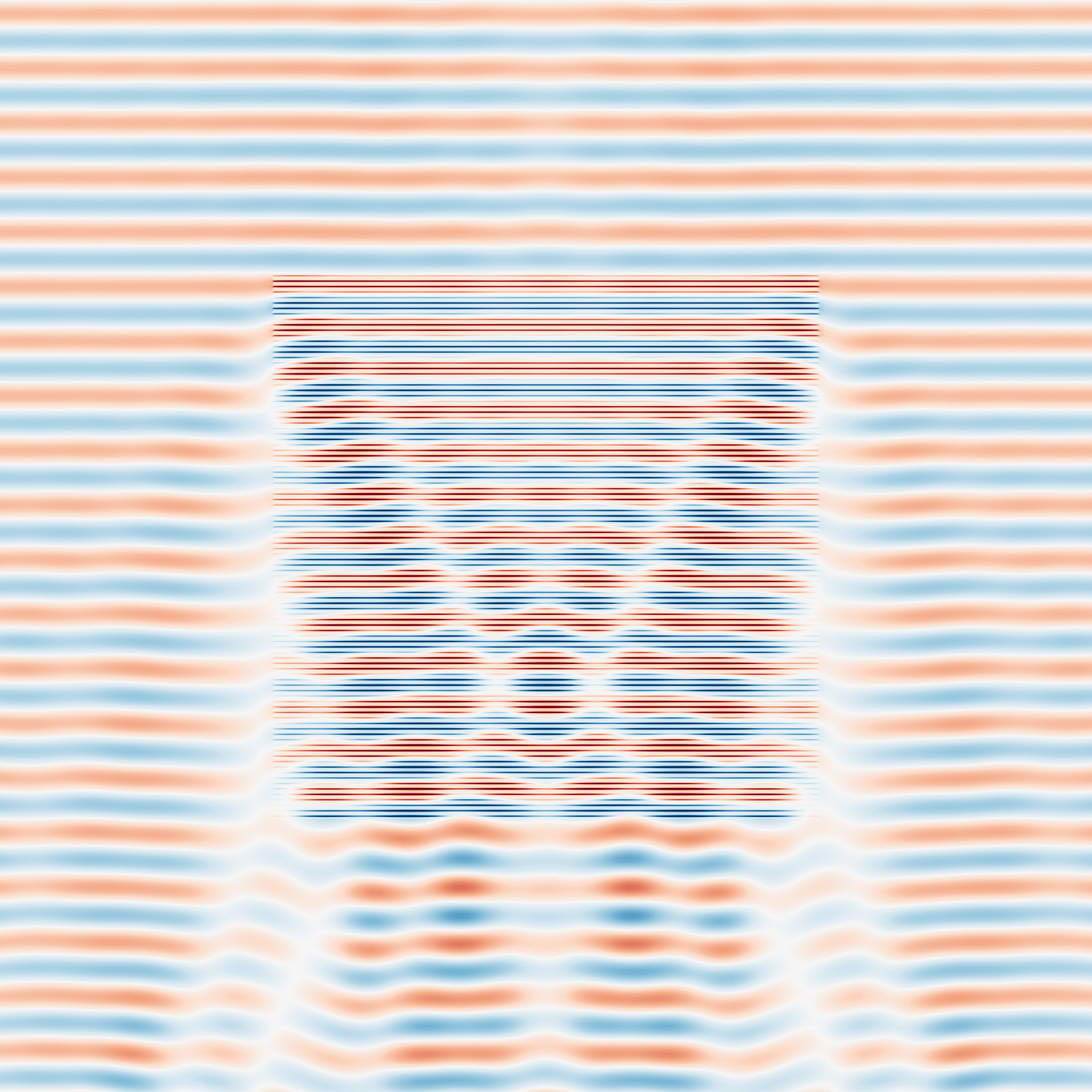}};

\draw[ultra thick,dotted] (-3.75,-3.75) rectangle (3.75,3.75);
\end{tikzpicture}
\subcaption{$\omega = 10\pi$ and $\eps = 0.02$}
\end{minipage}

%% file: figures/legend.tex
\centering
\begin{tikzpicture}
\begin{axis}
[
	hide axis,
	width=\linewidth,
	height=3cm,
	xmin=-3,
	xmax=3,
	ymin=-2,
	ymax= 1
]

\plot[solid                      ,mark=square*,black         ,domain=-2.5:-0.5,samples=5] {-1}
node[pos=0.5,anchor=north] {$\ell=1$ multiscale};
\plot[solid                      ,mark=*      ,blue          ,domain= 0.5: 2.5,samples=5] {-1}
node[pos=0.5,anchor=north] {$\ell=0$ multiscale};
\plot[dashed,mark options={solid},mark=square ,red           ,domain=-2.5:-0.5,samples=5] { 1}
node[pos=0.5,anchor=north] {$\ell=1$ coarse};
\plot[dashed,mark options={solid},mark=o      ,black!50!green,domain= 0.5: 2.5,samples=5] { 1}
node[pos=0.5,anchor=north] {$\ell=0$ coarse};

\end{axis}
\end{tikzpicture}

%% file: figures/figure_F1_N.tex
\begin{minipage}{.45\linewidth}
\begin{tikzpicture}
\begin{axis}
[
	width=\linewidth,
	xmode=log,
	ymode=log,
	xlabel={$N$},
	ymax=2,
	ymin=1.e-5
]

\plot[solid                      ,mark=square*,black         ] table[x=N,y=p1f] {figures/data/curve_1_0.200.txt};
\plot[solid                      ,mark=*      ,blue          ] table[x=N,y=p0f] {figures/data/curve_1_0.200.txt};
\plot[dashed,mark options={solid},mark=square ,red           ] table[x=N,y=p1c] {figures/data/curve_1_0.200.txt};
\plot[dashed,mark options={solid},mark=o      ,black!50!green] table[x=N,y=p0c] {figures/data/curve_1_0.200.txt};

\end{axis}
\end{tikzpicture}
\subcaption{$\eps = 0.200$}
\end{minipage}
\begin{minipage}{.45\linewidth}
\begin{tikzpicture}
\begin{axis}
[
	width=\linewidth,
	xmode=log,
	ymode=log,
	xlabel={$N$},
	ymax=2,
	ymin=1.e-5
]

\plot[solid                      ,mark=square*,black         ] table[x=N,y=p1f] {figures/data/curve_1_0.100.txt};
\plot[solid                      ,mark=*      ,blue          ] table[x=N,y=p0f] {figures/data/curve_1_0.100.txt};
\plot[dashed,mark options={solid},mark=square ,red           ] table[x=N,y=p1c] {figures/data/curve_1_0.100.txt};
\plot[dashed,mark options={solid},mark=o      ,black!50!green] table[x=N,y=p0c] {figures/data/curve_1_0.100.txt};

\end{axis}
\end{tikzpicture}
\subcaption{$\eps = 0.100$}
\end{minipage}

\begin{minipage}{.45\linewidth}
\begin{tikzpicture}
\begin{axis}
[
	width=\linewidth,
	xmode=log,
	ymode=log,
	xlabel={$N$},
	ymax=2,
	ymin=1.e-5
]

\plot[solid                      ,mark=square*,black         ] table[x=N,y=p1f] {figures/data/curve_1_0.050.txt};
\plot[solid                      ,mark=*      ,blue          ] table[x=N,y=p0f] {figures/data/curve_1_0.050.txt};
\plot[dashed,mark options={solid},mark=square ,red           ] table[x=N,y=p1c] {figures/data/curve_1_0.050.txt};
\plot[dashed,mark options={solid},mark=o      ,black!50!green] table[x=N,y=p0c] {figures/data/curve_1_0.050.txt};

\end{axis}
\end{tikzpicture}
\subcaption{$\eps = 0.050$}
\end{minipage}
\begin{minipage}{.45\linewidth}
\begin{tikzpicture}
\begin{axis}
[
	width=\linewidth,
	xmode=log,
	ymode=log,
	xlabel={$N$},
	ymax=2,
	ymin=1.e-5
]

\plot[solid                      ,mark=square*,black         ] table[x=N,y=p1f] {figures/data/curve_1_0.020.txt};
\plot[solid                      ,mark=*      ,blue          ] table[x=N,y=p0f] {figures/data/curve_1_0.020.txt};
\plot[dashed,mark options={solid},mark=square ,red           ] table[x=N,y=p1c] {figures/data/curve_1_0.020.txt};
\plot[dashed,mark options={solid},mark=o      ,black!50!green] table[x=N,y=p0c] {figures/data/curve_1_0.020.txt};

\end{axis}
\end{tikzpicture}
\subcaption{$\eps = 0.020$}
\end{minipage}

\begin{minipage}{.45\linewidth}
\begin{tikzpicture}
\begin{axis}
[
	width=\linewidth,
	xmode=log,
	ymode=log,
	xlabel={$N$},
	ymax=2,
	ymin=1.e-5
]

\plot[solid                      ,mark=square*,black         ] table[x=N,y=p1f] {figures/data/curve_1_0.010.txt};
\plot[solid                      ,mark=*      ,blue          ] table[x=N,y=p0f] {figures/data/curve_1_0.010.txt};
\plot[dashed,mark options={solid},mark=square ,red           ] table[x=N,y=p1c] {figures/data/curve_1_0.010.txt};
\plot[dashed,mark options={solid},mark=o      ,black!50!green] table[x=N,y=p0c] {figures/data/curve_1_0.010.txt};

\end{axis}
\end{tikzpicture}
\subcaption{$\eps = 0.010$}
\end{minipage}
\begin{minipage}{.45\linewidth}
\begin{tikzpicture}
\begin{axis}
[
	width=\linewidth,
	xmode=log,
	ymode=log,
	xlabel={$N$},
	ymax=2,
	ymin=1.e-5
]

\plot[solid                      ,mark=square*,black         ] table[x=N,y=p1f] {figures/data/curve_1_0.005.txt};
\plot[solid                      ,mark=*      ,blue          ] table[x=N,y=p0f] {figures/data/curve_1_0.005.txt};
\plot[dashed,mark options={solid},mark=square ,red           ] table[x=N,y=p1c] {figures/data/curve_1_0.005.txt};
\plot[dashed,mark options={solid},mark=o      ,black!50!green] table[x=N,y=p0c] {figures/data/curve_1_0.005.txt};

\end{axis}
\end{tikzpicture}
\subcaption{$\eps = 0.005$}
\end{minipage}

%% file: figures/figure_F5_N.tex
\begin{minipage}{.45\linewidth}
\begin{tikzpicture}
\begin{axis}
[
	width=\linewidth,
	xmode=log,
	ymode=log,
	xlabel={$N$},
	ymax=2,
	ymin=1.e-5
]

\plot[solid                      ,mark=square*,black         ] table[x=N,y=p1f] {figures/data/curve_5_0.200.txt};
\plot[solid                      ,mark=*      ,blue          ] table[x=N,y=p0f] {figures/data/curve_5_0.200.txt};
\plot[dashed,mark options={solid},mark=square ,red           ] table[x=N,y=p1c] {figures/data/curve_5_0.200.txt};
\plot[dashed,mark options={solid},mark=o      ,black!50!green] table[x=N,y=p0c] {figures/data/curve_5_0.200.txt};

\end{axis}
\end{tikzpicture}
\subcaption{$\eps = 0.200$}
\end{minipage}
\begin{minipage}{.45\linewidth}
\begin{tikzpicture}
\begin{axis}
[
	width=\linewidth,
	xmode=log,
	ymode=log,
	xlabel={$N$},
	ymax=2,
	ymin=1.e-5
]

\plot[solid                      ,mark=square*,black         ] table[x=N,y=p1f] {figures/data/curve_5_0.100.txt};
\plot[solid                      ,mark=*      ,blue          ] table[x=N,y=p0f] {figures/data/curve_5_0.100.txt};
\plot[dashed,mark options={solid},mark=square ,red           ] table[x=N,y=p1c] {figures/data/curve_5_0.100.txt};
\plot[dashed,mark options={solid},mark=o      ,black!50!green] table[x=N,y=p0c] {figures/data/curve_5_0.100.txt};

\end{axis}
\end{tikzpicture}
\subcaption{$\eps = 0.100$}
\end{minipage}

\begin{minipage}{.45\linewidth}
\begin{tikzpicture}
\begin{axis}
[
	width=\linewidth,
	xmode=log,
	ymode=log,
	xlabel={$N$},
	ymax=2,
	ymin=1.e-5
]

\plot[solid                      ,mark=square*,black         ] table[x=N,y=p1f] {figures/data/curve_5_0.050.txt};
\plot[solid                      ,mark=*      ,blue          ] table[x=N,y=p0f] {figures/data/curve_5_0.050.txt};
\plot[dashed,mark options={solid},mark=square ,red           ] table[x=N,y=p1c] {figures/data/curve_5_0.050.txt};
\plot[dashed,mark options={solid},mark=o      ,black!50!green] table[x=N,y=p0c] {figures/data/curve_5_0.050.txt};

\end{axis}
\end{tikzpicture}
\subcaption{$\eps = 0.050$}
\end{minipage}
\begin{minipage}{.45\linewidth}
\begin{tikzpicture}
\begin{axis}
[
	width=\linewidth,
	xmode=log,
	ymode=log,
	xlabel={$N$},
	ymax=2,
	ymin=1.e-5
]

\plot[solid                      ,mark=square*,black         ] table[x=N,y=p1f] {figures/data/curve_5_0.020.txt};
\plot[solid                      ,mark=*      ,blue          ] table[x=N,y=p0f] {figures/data/curve_5_0.020.txt};
\plot[dashed,mark options={solid},mark=square ,red           ] table[x=N,y=p1c] {figures/data/curve_5_0.020.txt};
\plot[dashed,mark options={solid},mark=o      ,black!50!green] table[x=N,y=p0c] {figures/data/curve_5_0.020.txt};

\end{axis}
\end{tikzpicture}
\subcaption{$\eps = 0.020$}
\end{minipage}

\begin{minipage}{.45\linewidth}
\begin{tikzpicture}
\begin{axis}
[
	width=\linewidth,
	xmode=log,
	ymode=log,
	xlabel={$N$},
	ymax=2,
	ymin=1.e-5
]

\plot[solid                      ,mark=square*,black         ] table[x=N,y=p1f] {figures/data/curve_5_0.010.txt};
\plot[solid                      ,mark=*      ,blue          ] table[x=N,y=p0f] {figures/data/curve_5_0.010.txt};
\plot[dashed,mark options={solid},mark=square ,red           ] table[x=N,y=p1c] {figures/data/curve_5_0.010.txt};
\plot[dashed,mark options={solid},mark=o      ,black!50!green] table[x=N,y=p0c] {figures/data/curve_5_0.010.txt};

\end{axis}
\end{tikzpicture}
\subcaption{$\eps = 0.010$}
\end{minipage}
\begin{minipage}{.45\linewidth}
\begin{tikzpicture}
\begin{axis}
[
	width=\linewidth,
	xmode=log,
	ymode=log,
	xlabel={$N$},
	ymax=2,
	ymin=1.e-5
]

\plot[solid                      ,mark=square*,black         ] table[x=N,y=p1f] {figures/data/curve_5_0.005.txt};
\plot[solid                      ,mark=*      ,blue          ] table[x=N,y=p0f] {figures/data/curve_5_0.005.txt};
\plot[dashed,mark options={solid},mark=square ,red           ] table[x=N,y=p1c] {figures/data/curve_5_0.005.txt};
\plot[dashed,mark options={solid},mark=o      ,black!50!green] table[x=N,y=p0c] {figures/data/curve_5_0.005.txt};

\end{axis}
\end{tikzpicture}
\subcaption{$\eps = 0.005$}
\end{minipage}

%% file: figures/figure_eps.tex
\begin{minipage}{.45\linewidth}
\begin{tikzpicture}
\begin{axis}
[
	width=\linewidth,
	xmode=log,
	ymode=log,
	xlabel={$\eps$},
	x dir=reverse,
	ymax=2,
	ymin=1.e-3
]

\plot[solid                      ,mark=square*,black         ] table[x=eps,y=p1f] {figures/data/curve_1_0032.txt};
\plot[solid                      ,mark=*      ,blue          ] table[x=eps,y=p0f] {figures/data/curve_1_0032.txt};
\plot[dashed,mark options={solid},mark=square ,red           ] table[x=eps,y=p1c] {figures/data/curve_1_0032.txt};
\plot[dashed,mark options={solid},mark=o      ,black!50!green] table[x=eps,y=p0c] {figures/data/curve_1_0032.txt};

\plot[dotted,domain=0.005:0.05] {0.5*x};

\SlopeTriangle{0.7}{-0.1}{0.15}{-1}{$\eps$}{}

\end{axis}
\end{tikzpicture}
\subcaption{$\omega = 2\pi$ with $N=32$}
\end{minipage}
\begin{minipage}{.45\linewidth}
\begin{tikzpicture}
\begin{axis}
[
	width=\linewidth,
	xmode=log,
	ymode=log,
	xlabel={$\eps$},
	x dir=reverse,
	ymax=2,
	ymin=1.e-3
]

\plot[solid                      ,mark=square*,black         ] table[x=eps,y=p1f] {figures/data/curve_5_0128.txt};
\plot[solid                      ,mark=*      ,blue          ] table[x=eps,y=p0f] {figures/data/curve_5_0128.txt};
\plot[dashed,mark options={solid},mark=square ,red           ] table[x=eps,y=p1c] {figures/data/curve_5_0128.txt};
\plot[dashed,mark options={solid},mark=o      ,black!50!green] table[x=eps,y=p0c] {figures/data/curve_5_0128.txt};

\plot[dotted,domain=0.005:0.02] {5*x};

\SlopeTriangle{0.7}{-0.1}{0.45}{-1}{$\eps$}{}

\end{axis}
\end{tikzpicture}
\subcaption{$\omega = 10\pi$ with $N=128$}
\end{minipage}